\documentclass[12pt]{amsart}

\usepackage{amssymb,amsmath,amsthm}  
\usepackage{graphicx}
\usepackage{bm}
\usepackage{enumerate}
\usepackage[shortlabels]{enumitem}
\usepackage{braket}
\usepackage[utf8]{inputenc}
\usepackage{amsaddr}
\usepackage{amscd}
\usepackage[dvipsnames]{xcolor}
\usepackage[mathscr]{eucal}
\usepackage{epstopdf}
\usepackage{breakcites}
\usepackage{hyperref}
\usepackage{xurl}
\hypersetup{breaklinks=true}
\usepackage{subcaption}
\usepackage{cleveref}
\usepackage{relsize}
\usepackage{cases}

\hypersetup{
  colorlinks   = true, 
  urlcolor     = PineGreen, 
  linkcolor    = PineGreen, 
  citecolor   = PineGreen 
}

 \allowdisplaybreaks
   
\numberwithin{equation}{section}
\newtheorem{theorem}{Theorem}[section] 
\newtheorem{definition}[theorem]{Definition}
\newtheorem{lemma}[theorem]{Lemma} 
\newtheorem{cor}[theorem]{Corollary}

\newtheorem{remark}[theorem]{Remark}
\newtheorem{proposition}[theorem]{Proposition}

\newcommand{\tu}{\textup}

\newcommand{\dd}{\displaystyle}

\newcommand{\e}{\epsilon}
\usepackage[numbers]{natbib}
\usepackage[letterpaper]{geometry}
\usepackage{pifont}
\usepackage{float}
\usepackage{color}
\usepackage{adjustbox}
\usepackage{diagbox}
\usepackage{multirow}
\usepackage[makeroom]{cancel}
\usepackage{array}
\usepackage{mathrsfs}
\usepackage{tikz}
\usepackage{comment}
\usepackage{csquotes}

\graphicspath{{./Figures/}}

\usepackage{amsfonts}

\usepackage{color}
\usepackage{adjustbox}
\usepackage{diagbox}
\definecolor{black}{rgb}{0,0,0}

\definecolor{red}{rgb}{1,0,0}

\definecolor{blue}{rgb}{0,0,1}

\newcommand{\tm}[1]{{\color{red}{#1}}}

\numberwithin{equation}{section}

\newcommand{\dr}{ \mathrm{d}r}
\newcommand{\ds}{ \mathrm{d}s}
\newcommand{\dx}{ \mathrm{d}x}
\newcommand{\dy}{ \mathrm{d}y}
\newcommand{\dt}{ \mathrm{d}t}
\newcommand{\dw}{ \mathrm{d}w}
\newcommand{\dz}{ \mathrm{d}z}
\newcommand{\du}{ \mathrm{d}u}
\newcommand{\dn}{ \mathrm{d}\nu}
\newcommand{\dth}{ \mathrm{d}\theta}
\newcommand{\dm}{ \mathrm{d}m}
\newcommand{\dtau}{ \mathrm{d}\tau}
\newcommand{\dxi}{ \mathrm{d}\xi}

\newcommand{\beq}{\begin{equation}}
\newcommand{\eeq}{\end{equation}}
\newcommand{\beqq}{\begin{equation*}}
\newcommand{\eeqq}{\end{equation*}}
\newcommand{\beqas}{\begin{eqnarray*}}
\newcommand{\eeqas}{\end{eqnarray*}}
\newcommand{\bsp}{\begin{split}}
\newcommand{\eesp}{\end{split}}

\makeatother

\usepackage[english]{babel}

\title[H\MakeLowercase{ydrodynamic limit of the} 
K\MakeLowercase{uramoto}--S\MakeLowercase{akaguchi equation with inertia and noise effects}]
{\textsf{\LARGE  H\MakeLowercase{ydrodynamic limit of the} 
K\MakeLowercase{uramoto}--S\MakeLowercase{akaguchi equation with inertia and noise effects}}}
\author[T\MakeLowercase{ina} M\MakeLowercase{ai}]{\Large T\MakeLowercase{ina} M\MakeLowercase{ai$^{a,b,c,*}$}}
        
\begin{document}
%
%
%
%
\begin{abstract}
 
We consider the Kuramoto--Sakaguchi--Fokker--Planck equation (namely, parabolic Kuramoto--Sakaguchi, or Kuramoto--Sakaguchi equation, which is a nonlinear parabolic integro-differential equation) with inertia and white noise effects. 
We study the hydrodynamic limit of this Kuramoto--Sakaguchi equation.  During showing this main result, as a support, we also prove a Hardy-type inequality over the whole real line.
%
\end{abstract}

\date{\today}
\maketitle


\vskip10pt

%
%

\vfill

\noindent 
$^{*}$Corresponding author: 
\textit{Tina Mai};  
$^a$Institute of Research and Development, Duy Tan University, Da Nang, 550000, Vietnam; $^b$Faculty of Natural Sciences, Duy Tan University, Da Nang, 550000, Vietnam; $^c$Department of Mathematics and Institute for Scientific Computation, Texas A\&M University, College Station, TX 77843, USA;\\ \texttt{maitina@duytan.edu.vn}; \texttt{tinamai@tamu.edu}  

\newpage

\tableofcontents

\section{Introduction}
Synchronization is the coordination coupled oscillators.
It compromises natural 
frequencies of the oscillators to a common frequency (when these oscillators weakly interact with one another).  This process appears in a wide range of phenomena observed in 
the real world,
such as human movement, image in sound film, neural assemblies \cite{2010benoit-syn}; and it has drawn remarkable attentions because of its biological, social, and engineering applications \cite{2020difflim, 2022choihd}. 
In this paper, we focus on deriving the hydrodynamic 
limit of the stochastically perturbed Kuramoto--Sakaguchi equation, among mathematical models describing the synchronization (thanks to \cite{KM1,KM2,KM3} and 
the authors). Our setting is a large ensemble of Kuramoto oscillators having finite inertia, in a random media with white noise effects.
Let 
$F = F (\theta, w, \nu, t) \geq 0$ 
be the one-oscillator
probability density function (at phase $\theta\,,$ with frequency $w$ and natural frequency $\nu$).  The phase-space-temporal evolution of $F$ is governed by the Cauchy problem in terms of the
following \textit{Vlasov--Fokker--Planck type equation}, more specifically, \textit{Kuramoto--Sakaguchi equation} (as a nonlinear parabolic integro-differential equation) \cite{2018wellposed, 2020difflim}: 

\begin{equation}\label{origin}
 \partial_t F + \partial_{\theta} (w F) + \partial_w(\mathcal{A}[F]F) 
 = \frac{\sigma}{m^2} \partial_w^2 F\,,
\end{equation}
with 
\[(\theta, w, \nu, t) \in \mathbb{T} \times \mathbb{R} \times \mathbb{R} \times 
\mathbb{R_{+}}\,, \quad \mathbb{T}: = [0, 2\pi) \,.\]
The distribution of natural frequencies obeys the probability density function $g(\nu) \geq 0\,,$ and $\nu \in \text{supp} \, g(\nu)$ \cite{2019robust}.
%

Here, $\mathcal{A}[F]$ is the synchronizing forcing term:
\begin{equation}\label{A}
\mathcal{A}[F] (\theta, w, \nu, t) : = \frac{1}{m} (-w + \nu + K \mathcal{S}[F])\,,
\end{equation}
where,
\begin{equation}\label{S}
 \mathcal{S}[F] := \iiint\limits_{\mathbb{R} \; \mathbb{R} \; \mathbb{T}} 
 \text{sin}
 \left(\theta_{*} - \theta\right) F(\theta_{*}, w_{*},\nu_{*},t) \, g(\nu_*) \,
  \dth_{*} \, \dw_{*} \, \dn_{*} \,.
\end{equation}


\noindent Note that the multiple integral in \eqref{S} may also be expressed as
\[\int_{\mathbb{T} \times \mathbb{R} \times \mathbb{R}}  \text{sin}
 \left(\theta_{*} - \theta\right) F(\theta_{*}, w_{*},\nu_{*},t) \, g(\nu_*) \, 
 \dth_{*} \, \dw_{*} \, \dn_{*} \,.\]  
 The \textit{torus} is the set $\mathbb{T} = [0,2\pi)\,,$ where $0$ and $2\pi$ are related to one another, a point $\theta \in \mathbb{T}$ is an \textit{angle}, and a connected subset of $\mathbb{T}$ is an \textit{arc} \cite{12circlepath, 12circle}.  

%
%
The initial, boundary (periodicity), and decay conditions are given respectively by \cite{2020difflim} 
\begin{align}\label{ini}
\begin{split}
&F(\theta, w, \nu, 0) = F^{\text{in}} (\theta, w, \nu) > 0\,, \quad (\theta, w, \nu) \in \mathbb{T} \times \mathbb{R} \times \mathbb{R}\,,\\
&\iint\limits_{\mathbb{R} \; \mathbb{T}}  F^{\text{in}} (\theta, w, \nu) \, \dth \, \dw = 1 \,, \quad
\iiint\limits_{\mathbb{R} \; \mathbb{R} \; \mathbb{T}} F^{\text{in}} (\theta, w, \nu) \, \dth \, \dw \, \dn = 1\,,\\
 &F(0, w, \nu, t) = F(2\pi, w, \nu, t) \,, \\
 &\lim_{|w| \to \infty} F(\theta,w,\nu,t) = 0\,.
 \end{split}
\end{align}

\begin{remark}\label{0pi}
Following \cite{12circlepath, 12circle}, toward a rigorous definition of the \textit{difference} between angles (that is, points on the circle $\mathbb{S}^1$), we restrict our attention to angles $\theta_*$ contained in the domain
\begin{equation}\label{size}
\mathcal{D} = \{\theta_* \in [0,2\pi) \, | \, (\theta_* - \theta) \in \mathbb{D} = (0,\pi) \} \subset \mathbb{T} \,.
\end{equation}
To get this restriction and ensure synchronization, the initial phases are assumed to lie in an open half-circle \cite{12circlepath,12plane,KM1}: for phases $\theta_1\,, \theta_2$ with $|\theta_2 - \theta_1| < \pi\,,$ the \textit{angular difference} $\theta_2 - \theta_1$ is defined as the number belonging to $(-\pi, \pi)$ whose magnitude is equal to the geodesic distance $|\theta_2 - \theta_1|$ and having positive value if and only if the counterclockwise path length from $\theta_1$ to $\theta_2$ on $\mathbb{T}$ is less than the clockwise path length \cite{12circlepath}.  
\end{remark}

The Maxwellian is given by \cite{Kramers1, Taylor1,2018wellposed}:
\begin{equation}\label{Maxw}
 M(w,\nu): = \frac{1}{2\pi} \sqrt{\frac{m}{2\pi \sigma}} \text{ exp } 
 \left(- \frac{m}{2\sigma} (w- \nu)^2\right)\,.
\end{equation}

Equations (\ref{origin}), (\ref{A}), and (\ref{S}) should be supplemented with 
appropriate initial condition, boundary data ($F$ is $2\pi$ periodic in $\theta$), 
and has suitable decay behavior for $F$ (when $w \to \pm \infty$) as in (\ref{ini}), plus the following normalization condition (see \cite{e, Taylor1}, for instance):
\begin{equation}\label{renorm}
 \iint\limits_{\mathbb{T} \; \mathbb{R}} F(\theta, w, \nu, t) \,  \dw \, \dth = 1\,.
\end{equation}

Here, $F$ is defined such that at each time $t$ and for each $\nu$, the fraction of 
 oscillators with phases between $\theta$ and 
 $\theta + d\theta$ 
 and frequencies between $w$ and $w + dw$ is given by $F(\theta, w, \nu, t) \, \dth \, \dw$ 
 \cite{ff, Kramers1} (where natural frequency $\nu$ is used instead of frequency $w$).

In \cite{2020difflim}, the authors showed the global well-posedness (existence and uniqueness) of weak solutions to the Cauchy problem \eqref{origin}--\eqref{ini} in any finite-time interval.  Also, \cite{2018wellposed} presented the global-in-time existence and uniqueness of strong solutions around a phase-homogeneous solution for the problem \eqref{origin}--\eqref{ini} with supplemented conditions.  Therefore, we assume here the existence and uniqueness of solutions (which are as regular as necessary) for the system \eqref{origin}--\eqref{ini}.  

In this paper, we mainly study the hydrodynamic limit of the Kuramoto--Sakaguchi equation with inertia and noise effects \eqref{origin}--\eqref{ini}, following the methodology in \cite{PD1} by Degond, Dimarco, and Mac. As with the books \cite{L} by Saint-Raymond and \cite{golsehd} by Golse, the terminology hydrodynamic limit stands for a macroscopic description (such as the basic partial differential equations \eqref{hl1}--\eqref{hl3} of hydrodynamics) derived from a microscopic representation of matter (such as the microscopic Kuramoto--Sakaguchi equation \eqref{origin}--\eqref{ini} describing the oscillators' synchronization).  While deriving this hydrodynamic limit (Theorem \ref{thm}), as an assistance, we also prove a Hardy-type inequality \eqref{hardyineq} over the whole real line.



Our paper has the following organization.  Section \ref{form} contains preliminary material 
regarding the Kuramoto--Sakaguchi model with inertia and white noise effects.  In Section 
\ref{hl0}, its first five subsections form a preparation for the later proof of our main 
Theorem \ref{thm} (hydrodynamic limit of the model) in Subsection \ref{proofthm}.  
In particular, Subsection \ref{non} is devoted 
to nondimensionalizing the system of equations \eqref{origin}--\eqref{ini}.  In Subsection \ref{scale}, we rescale the nondimensionalized system from 
microscopic scales of time and phase 
to their macroscopic ones.  The relevant scaling is $\e$, a hydrodynamic one, that is, 
$\e$ equals the ratio of the microscopic time scale (the mean time between 
collisions or meetings in phase between the $i$th and $j$th oscillators \cite{KM1}) to 
the macroscopic observation time; then, our main Theorem \ref{thm} of hydrodynamic limit 
$\e \to 0$ is stated there.  For the purpose of proving Theorem \ref{thm}, we first show in Subsection \ref{hard} a Hardy-type inequality over the whole real line.  Then, Subsection \ref{qprop} includes some properties of the collision operator $Q$.  The 
Generalized Collision Invariants (GCI) of $Q$ are characterized in Subsection \ref{gci0}.  
In the last Subsection \ref{proofthm}, we give a proof for our major Theorem \ref{thm}.  In the appendices, we respectively present some preliminary results about absolute continuity (Appendix \ref{hardytype}), improper Riemann integrals and Lebesgue integrals (Appendix \ref{appii}), and a Hardy's inequality with weights over a half-open interval (Appendix \ref{hiw1}).  

\section{Preliminaries}\label{form}
In this section, we review the Kuramoto model under inertia and noise effects. 
Let $\theta_i = \theta_i(t)$ be the
phase of the $i$th oscillator with natural frequency $\nu_i$ 
in the presence of inertia and noise
effects. Then, the point $e^{\sqrt{-1} \theta_i}$ can be regarded as the location of the
$i$th point rotator on
the unit circle $\mathbb{S}^1$. In this situation, the dynamics of the $i$th phase 
$\theta_i$ is governed by the
following system of second-order ODEs \cite{Kramers1,Taylor1}: 
\begin{equation}\label{odes}
 m \, \ddot \theta_i + \dot \theta_i =
 \nu_i +\frac{K}{N} \sum_{j=1}^N \text{ sin }(\theta_j -\theta_i) 
 + \sqrt{2\sigma} \, \eta_i\,,
\end{equation}
where $N$ is the total number of interacting oscillators, while 
$m\,, K\,,$ and $\sigma$ are respectively the strength of inertia, coupling, and noise (also called noise strength or noise intensity), which are constants.  
The noise satisfies the (Gaussian) white noise condition
\begin{equation}\label{noise}
 \langle \eta_i(t) \rangle = 0, \quad \langle \eta_i(s) \, \eta_j(t) \rangle = 
 \delta_{ij} \, \delta(t-s)\,.
\end{equation}
Here, $\langle \cdot \rangle$ denotes the expected value.


The system (\ref{odes}) can be rewritten 
as a system of the first-order ODEs for 
$(\theta_i, w_i)$:
\begin{align}\label{ode2}
\begin{split}
\dot \theta_i &= w_i\,, \\
\dot w_i &= \dfrac{1}{m} 
\left( - w_i + \nu_i
+ \dfrac{K}{N} \sum_{j=1}^N \text{sin} 
(\theta_j - \theta_i)\right) + \dfrac{\sqrt{2\sigma}}{m} \, \eta_i \,.
\end{split}
\end{align}

As in \cite{Kramers1, KM2, KM3}, equation (\ref{origin}) for $F$ can be derived from 
equation (\ref{odes}), by the formal thermodynamic limit ($N \to \infty$) using the standard 
BBGKY hierarchy (Bogoliubov--Born--Green--Kirkwood--Yvon hierarchy).
\section{Hydrodynamic limit}\label{hl0}

This section follows \cite{PD1, jose19, trivisa14}, mainly \cite{PD1}.  In the literature, the Vlasov limit for the Kuramoto model was obtained in \cite{lancellotti} by Lancellotti.  Then, for the problem \eqref{origin}--\eqref{ini}, a diffusion limit has been derived by Ha, Shim, and Zhang \cite{2020difflim}.  Also, in \cite{2022choihd}, Choi has investigated the hydrodynamic limit of the Kuramoto model (without noise).  In this section, we study the hydrodynamic limit of the Kuramoto--Sakaguchi model \eqref{origin}--\eqref{ini} with inertia and white noise effects.  

\subsection{Nondimensionalization}\label{non}
Our study of the hydrodynamic limit is motivated by \cite{PD1,hlk2,GJV,L,DL16,BGL,Levermore,tm_entropic,alonsohd}.  Toward this study, we should first nondimensionalize equations (\ref{origin}), (\ref{A}), and 
(\ref{S}) in an appropriate way \cite{e}:
\[t=t_0 \,\hat{t}, \quad w = w_0 \, \hat{w}\,.\]

The main idea is that the force term and diffusion in 
frequency space should be dominant.  
Then, the density $F$ rapidly reaches local equilibrium in frequency.  This means that $\nu$ and 
$K$ should have the same order as $w\,,$ and that the terms $m^{-1} \partial_w(wF) \text{ and } 
(\sigma/m^2) \partial_{w}^2 F$ should be of the same order (subscripts imply the partial derivative with respect to the
corresponding variable).  
If we call $w_0$ a typical unit of frequency, then the latter 
balance yields the frequency scale

\[w_0 = \sqrt{\frac{\sigma}{m}}\,.\]

\noindent The ratio of $w \partial_{\theta} F$ to $m^{-1} \partial_w (w F)$ is of the order

\[\alpha = \sqrt{\sigma m}\,.\]

\noindent Lastly, we will choose the time unit so that $\partial_t F$ and 
$w \partial_{\theta} F$ are of the 
same order.  This choice yields 
a time unit

\[t_0 = \frac{1}{w_0} = \sqrt{\frac{m}{\sigma}}\,.\]

\noindent The normalization condition and the definition \eqref{S} control that $g(\nu)$ and the density $F$ are to be measured in units of $1/w_0$ too.  The phase $\theta$ is already dimensionless.

With these choices, the time unit is the time needed by an oscillator to adjust its velocity due to weak interactions with other oscillators (or mean interaction
time) and the phase unit is the mean angle traveled by the oscillators during the
mean interaction time, that is, the mean-free path \cite{PD1}.

In summary, $w\,, \nu$ and $K$ are measured in units of the thermal velocity $w_0\,;$ while $F\,,$ $g(\nu)$ and $t$ are
measured in units of 
the reciprocal frequency $1/w_0\,.$  Then, equation \eqref{origin} (with \eqref{A}--\eqref{ini}) becomes
\begin{align}\label{primes}
 \begin{split}
\sqrt{\frac{\sigma}{m}} \frac{\partial F}{\partial \hat{t}} + 
\sqrt{\frac{\sigma}{m}} \hat{w} \frac{\partial F}{\partial \theta} 
&= \frac{1}{m} \sqrt{\frac{m}{\sigma}} \frac{\partial}{\partial \hat{w}} 
\left(\left(\sqrt{\frac{\sigma}{m}} \hat{w} - \sqrt{\frac{\sigma}{m}} \hat{\nu} - 
\sqrt{\frac{\sigma}{m}} \hat{K} \mathcal{S}[F]\right)F\right) \\
& \quad + \frac{\sigma}{m^2} \frac{m}{\sigma} \frac{\partial^2 F}{\partial \hat{w}^2}\,.
 \end{split}
\end{align}

\noindent Dropping the hats for clarity, we obtain the dimensionless Fokker-Planck equation 
\begin{equation}\label{FP}
 \alpha(\partial_t F + w \, \partial_{\theta} F) = 
 \frac{\partial}{\partial w} \left(\frac{\partial F}{\partial w} +
 ( w - \nu - K \mathcal{S}[F])F\right)\,.
\end{equation}
This equation will be solved with equations (\ref{S}), (\ref{renorm}), 
initial condition, boundary data ($F$ is $2\pi$ periodic in $\theta$), 
and decay condition for $F$ (when $w \to \pm \infty$) as in (\ref{ini}).  

\subsection{Scaling}\label{scale}

So far, the chosen time and phase scales are microscopic ones.  
We are now interested by a description of 
the system at macroscopic scales, where scales are described by the units 
\[ t'_0 = \frac{t_0}{\e}, \quad \theta'_0 = \frac{\theta_0}{\e}\,,\]
in which $\e \ll 1$ is a parameter, $t_0\,, \theta_0$ are microscopic units, and 
$ t'_0\,,  \theta'_0$ are macroscopic units.

By changing these units, we correspondingly change the variables $t\,, \theta$ and 
the unknown 
$F$ to the new variables and unknown
\[ t' = \e t \,, \quad  \theta' = \e \theta\,,\]
where, thanks to \cite{expand1}, we define 
\begin{equation}\label{fe}
 f^{\e} =  f^{\e}( \theta', w, \nu,  t') = F = F(\theta, w, \nu, t)\,.
\end{equation}

After the change to macroscopic variable $ t', \theta'$, (\ref{FP}) becomes
\[\alpha \left(\frac{\partial f^{\e}}{\partial(\e^{-1}t')} + w \frac{\partial f^{\e}}{\partial(\e^{-1}\theta')} 
\right)
= \frac{\partial}{\partial w} \left(\frac{ f^{\e}}{\partial w} 
+ (w - \nu - K \mathcal{S}[ f^{\e}])  f^{\e}\right)\,.\]
Dropping the primes 
for simplicity, the equation then reads 
\begin{equation}\label{perturbation}
 \e \alpha ( \partial_t f^{\e} + w \, \partial_{\theta}f^{\e}) = 
   \partial_w\left(\left( w - \nu 
 - K J^{\e}_{f^{\e}}\right)f^{\e}\right) +  \partial_w^2 f^{\e}\,,
\end{equation} 
where
\begin{equation}\label{J}
 J^{\e}_f (\theta,t) =
 \iiint\limits_{\mathbb{R} \; \mathbb{R} \;  \e \mathbb{T}} 
 \text{sin}
\left(\frac{\theta_{*} - \theta}{\e}\right) f (\theta_{*}, w_{*}, \nu_{*}, t) \,  g(\nu_*) \,
 \frac{\dth_{*}}{\e} \, \dw_{*} \, \dn_{*}\,.
\end{equation}
To derive such an expression for $J^{\e}_f\,,$ we first rewrite (\ref{S}) into
\begin{align}\label{Snew}
 \mathcal{S}[F] &= \iiint\limits_{\mathbb{R} \; \mathbb{R} \; \mathbb{T}} 
 \text{sin}
 \left(\tau - \theta\right) F(\tau, w_{*},\nu_{*},t) \, g(\nu_*) \,
 \dtau \, \dw_{*} \, \dn_{*}\nonumber\\
 &= \iiint\limits_{\mathbb{R} \; \mathbb{R} \; \mathbb{T}} 
 \text{sin}
 \left(\frac{\e \tau - \theta'}{\e}\right) F(\tau, w_{*},\nu_{*},t) \, g(\nu_*) \,
 \dtau \, \dw_{*} \, \dn_{*}\,.
\end{align}

\noindent Changing $\tau$ to the macroscopic variable $\theta_{*} = \tau' = \e \tau$, then 
from (\ref{fe}), we get
\begin{align}\label{fee}
 J^{\e}_{f^{\e}} (\theta',t') &= \mathcal{S}[ f^{\e}(\tau',w,\nu,t')]\nonumber\\
 &= \iiint\limits_{\mathbb{R} \; \mathbb{R} \;   \mathbb{T}} 
 \text{sin}
 \left(\frac{\e \tau - \theta'}{\e}\right)  f^{\e} (\tau', w_{*}, \nu_{*}, t') \, g(\nu_*) \, 
 \dtau \, \dw_{*} \, \dn_{*} \nonumber \\
  &= \iiint\limits_{\mathbb{R} \; \mathbb{R} \; \e  \mathbb{T}} 
 \text{sin}
 \left(\frac{\theta_{*} - \theta'}{\e}\right)  f^{\e} (\theta_{*} , w_{*}, \nu_{*}, t') \, g(\nu_*)
 \frac{\dth_{*}}{\e} \, \dw_{*} \, \dn_{*} \nonumber \\
 &= \iiint\limits_{\mathbb{R} \; \mathbb{R} \; \e \mathbb{T}} 
 \text{sin}
 \left(\frac{\theta_{*} - \theta'}{\e}\right) f^{\e} (\theta_{*}, w_{*}, \nu_{*}, t')  \, g(\nu_*) \,
 \frac{\dth_{*}}{\e} \, \dw_{*} \, \dn_{*} \,.
\end{align}
Again, dropping the primes for simplicity, we obtain the desired result:
\begin{equation}\label{Jn}
 J^{\e}_{f^{\e}} (\theta,t) =
 \iiint\limits_{\mathbb{R} \; \mathbb{R} \; \e  \mathbb{T}} 
 \text{sin}
 \left(\frac{\theta_{*} - \theta}{\e}\right) f^{\e} (\theta_{*}, w_{*}, \nu_{*}, t) \, g(\nu_*) \,
 \frac{\dth_{*}}{\e} \, \dw_{*} \, \dn_{*}\,.
\end{equation}


It is important to realize that $J^{\e}_{f^{\e}}$ now depends on $\e$ and can be easily expanded in terms of $\e$ as follows 
(as the non-locality only appears at high order 
\cite{PD1,hlk2,expand1}).
\begin{lemma}\label{Oe2}
Assume that $f$ is as regular as necessary.  
We then have the expansion
 \begin{equation}\label{exOe2}
  J^{\e}_f (\theta,t) = \Phi_f (\theta,t) + \mathcal{O}(\e)\,,
 \end{equation}
 where 
\begin{equation}\label{K}
 \Phi_f (\theta,t) = 
\iint\limits_{\mathbb{R} \; \mathbb{R}} 
  f(\theta, w_{*}, \nu_{*}, t) \, g(\nu_*) \, \dw_{*} \, \dn_{*}\,.
 \end{equation}
\end{lemma}

\begin{proof}
The proof of this Lemma is fundamental and based on Taylor expansion (see \cite{PD1,hlk2}, for instance).  Another elegant approach is using Fourier transform.

The first approach (using Taylor expansion) is presented here in details.  
We will make a change of variable $\theta_{*} = \theta + \e \xi$, where $\xi \in (0, \pi)=\mathbb{D}$ 
(thanks to Remark \ref{0pi}) and will use the normalization
\begin{equation}\label{normxi}
 \int\limits_{\mathbb{D}}  \text{sin}
 \left(\xi\right) \dxi = 1\,.
\end{equation}
More specifically, let us expand $f^{\e}$ at the first 
order in $\e$ in (\ref{Jn}) to obtain
\begin{align}\label{expande}
  J^{\e}_{f^{\e}} (t, \theta) 
  &= \iiint\limits_{\mathbb{R} \times \mathbb{R} \times \mathbb{D}} 
 \text{sin}
 \left(\xi\right) 
 (f^{\e} (\theta + \e \xi, w_{*}, \nu_{*}, t)) \, g(\nu_*) \, 
 \frac{\e \dxi}{\e} \, \dw_{*} \, \dn_{*}\nonumber\\
 & = \iiint\limits_{\mathbb{R} \times \mathbb{R} \times \mathbb{D}}
 \text{sin}
\left(\xi\right) 
\left(f^{\e} (\theta, w_{*}, \nu_{*}, t) + \e \, \xi \,  \partial_{\theta}
f^{\e} (\theta, w_{*}, \nu_{*}, t) + \mathcal{O}(\e^2)\right) g(\nu_*) \, 
\dxi \, \dw_{*} \, \dn_{*}\nonumber\\
 & = \left(\int\limits_{\mathbb{D}}  \text{sin}
 \left(\xi\right) \dxi 
\iint\limits_{\mathbb{R} \times \mathbb{R}} f^{\e} (\theta, w_{*}, \nu_{*}, t) \, g(\nu_*) \,
  \dw_{*} \, \dn_{*}\right) + \mathcal{O}(\e)\nonumber\\
  &= \iint\limits_{\mathbb{R} \times \mathbb{R}} f^{\e} (\theta, w_{*}, \nu_{*}, t) \, g(\nu_*) \,
  \dw_{*} \, \dn_{*} + \mathcal{O}(\e)\nonumber\\
  & = \Phi_{f^{\e}} (\theta,t) + \mathcal{O}(\e)\,.
\end{align}
\end{proof}

The meaning of Lemma \ref{Oe2} is that up to $\mathcal{O} (\e)$ terms, 
the interaction force $J^{\e}_{f}$ is given by a local expression, 
involving only the distribution function $f$ at phase $\theta$.  
The quantity $\Phi_f (\theta, t)$ is the local oscillator flux at phase 
$\theta$ and time $t$.  By contrast, the expression of $J^{\e}_{f}$ 
in (\ref{Jn}) is 
non-local in phase: it involves 
a convolution of $f$ with respect to the non-local function sine. 

\begin{remark}\label{r1}
 We now omit the $\mathcal{O}(\e)$ terms in (\ref{exOe2}) as they do not contribute to 
 the hydrodynamic limit 
 at the leading order (which is what we are interested in) \cite{PD1}.
\end{remark}

Let us denote
\begin{equation}\label{mathcalF}
 \mathcal{F}[f] := w - \nu - K\Phi_f (\theta,t)\,.
\end{equation}

\noindent From now on, we write 
$\mathcal{F}[f^{\e}]$ for $\mathcal{F}^{\e}_{f^{\e}}$ and write $Q$ for the collision operator
\begin{equation}\label{Q}
 Q(f) := \partial_w \left(\mathcal{F}[f]f\right) + \partial_w^2 f \,. 
\end{equation}
Thus, (\ref{perturbation}) is equivalent to
\begin{equation}\label{per2}
  \e \alpha (\partial_t f^{\e} + w\partial_{\theta}f^{\e}) = Q(f^{\e})\,.
\end{equation}

Before stating the main Theorem \ref{thm}, we need to recall the definition of the 
\textit{von Mises (VM)} distribution on the circle $\mathbb{S}^1$
(which is a 1-sphere in $\mathbb{R}^2\,,$ as a special case of the von Mises--Fisher distribution on the $(p-1)$-sphere in $\mathbb{R}^p$).

Given $\nu \in \text{supp} \, g(\nu)$, let 
\begin{equation}\label{PO}
P_{\nu} = P_{\nu}(\theta,t):=P(\theta,\nu,t)
\end{equation}
be an arbitrary function of
$\theta$ and $t\,,$ and denote 
\begin{equation}\label{P}
 P = P(\theta,t) := \int\limits_{\mathbb{R}} P_{\nu_*}(\theta,t) \, g(\nu_*)  \, \dn_*\,.
\end{equation}

 
\noindent Then, let
\begin{equation}\label{u}
 u = u(\theta, t) := K P (\theta,t)\,.
\end{equation}
As in \cite{e}, with the provided $\nu \in \text{supp} \, g(\nu)$, we denote 
\begin{equation}\label{V}
 V = V(\theta,\nu,t): = \nu + u =  \nu + K P (\theta,t)\,;
\end{equation}
and the expression of the \textit{von Mises} distribution is
\begin{align}\label{equi}
 M_{V}(w)  
 &= \frac{1}{\sqrt{2 \pi}} \text{ exp}\left(-\dfrac{1}{2} 
 (w - V)^2\right) \,.
\end{align}

The normalization condition is set as follows \cite{e}:
\begin{equation}\label{Pcond}
\int\limits_{\mathbb{T}} P_{\nu}(\theta,t) \, \dth = 1\,.
\end{equation}
Note that as defined in \eqref{PO}, $P_{\nu} = P_{\nu}(\theta,t):=P(\theta,\nu,t)$ is an arbitrary function of
$\theta$ and $t\,,$ except in \eqref{Pcond}.  For our main Theorem \ref{thm}, we further assume that 
\begin{equation}\label{pt}
\int\limits_{\mathbb{R}} \nu_{*} \, P_{\nu_*}(\theta,t) \, g(\nu_*) \,
 \dn_* = \hat{C}
 < + \infty\,,
\end{equation}
where $\hat{C}$ is a constant.
\noindent Also, we impose the normalization condition  
\begin{equation}\label{normM}
\int\limits_{\mathbb{R}} M_{V}(w) \, \dw = 1\,.
\end{equation}

Using \eqref{equi}, since 
\begin{equation}\label{m0}
\int\limits_{\mathbb{R}} (w - V) \, M_V(w)  \, \dw = 
\int\limits_{\mathbb{R}} \partial_w M_V(w)  \, \dw=0 \,,
\end{equation}
it follows from (\ref{normM}) and \cite{PD1} that the flux of the \textit{von Mises} distribution is given by 
\begin{align}\label{c1}
\int\limits_{\mathbb{R}} w \, M_V(w) \, \dw = 
  \int\limits_{\mathbb{R}} V M_{V}(w) \, \dw 
 = V\,.
\end{align}

When setting $\e = 0$ in (\ref{per2}), we obtain a simple equation to be solved together with 
equations (\ref{S}) and (\ref{renorm}).  Its solution is a displaced Maxwellian \cite{e}:
\begin{equation}\label{dmaxw}
 f^0:=f^0(\theta,w,\nu,t) = P_{\nu}(\theta, t) M_V(w)\,,
\end{equation}
which also corresponds to a particular form of the 
initial conditions \cite[Eq.~(9)]{e}. Later, in Lemma \ref{m2}, we will derive \eqref{dmaxw}, and it is called an equilibrium by Remark \ref{equildef}.
\bigskip
Now, being directly motivated by \cite{PD1}, our goal is to investigate the hydrodynamic limit 
$\e \to 0$ of (\ref{per2}), 
that is, to establish a set of hydrodynamic equations for 
$P_{\nu}(\theta,t)$ and $V(\theta,\nu,t)$.  More precisely, the main result of this 
paper is the following Theorem.
\begin{theorem}[\textbf{Hydrodynamic limit}]\label{thm}
 Consider equation \eqref{per2}. 
 We assume that the limit $\dd f^0 = \lim_{\e \to 0} f^{\e}$ exists and that the convergence is as 
 regular as necessary (that is, occurs in functional spaces which allow the rigorous 
 justification of all the computation below).  Then, we obtain
\begin{equation}\label{hl}
  f^0(\theta,w, \nu, t) = P_{\nu}(\theta,t) \, M_{V(\theta,\nu,t)} (w)\,,
 \end{equation}
where for any $(\theta,t)$, the function $\nu \in \mathbb{R} \mapsto 
P_{\nu}(\theta,t) \in \mathbb{R}_{+}$ belongs to $L^1(\mathbb{R})$ and meets the conditions (\ref{Pcond}) and has the first moment  as the assumption \eqref{pt}:
\begin{equation}\label{pt2}
\int\limits_{\mathbb{R}} \nu_{*} \, P_{\nu_*}(\theta,t) \, g(\nu_*) \,
 \dn_* = \hat{C}
 < + \infty\,,
\end{equation}
in which $\hat{C}$ is a constant; while as \eqref{V}, 
\[\dd V=V(\theta,\nu,t) = \nu + KP(\theta,t) = \nu + K \int\limits_{\mathbb{R}} \, P_{\nu_*}(\theta,t) \, g(\nu_*) \,  \dn_*\]
belongs to $\mathbb{R}\,,$ and the normalization (\ref{normM}) of $M_V$ holds.  The functions $P_{\nu}(\theta,t)$ and 
$V(\theta,\nu,t)$ satisfy the following system of hydrodynamic limit equations:
\begin{numcases}{}
 \partial_t P_{\nu} + \partial_{\theta}(V P_{\nu}) = 0 
 \quad \forall \nu \in \mathbb{R}\,, \label{hl1}\\
 P \, \partial_t V +  P \, (Y+KP) \, \partial_{\theta} V + \partial_{\theta} P = 0\,, \label{hl2}
\end{numcases}
with
\begin{equation}\label{hl3}
P(\theta,t) = \int\limits_{\mathbb{R}} P_{\nu_*}(\theta,t) \, g(\nu_*)  \, \dn_*\,, 
\quad P (\theta,t) \, Y(\theta,t)= \int\limits_{\mathbb{R}} 
 \nu_* \, P_{\nu_*}(\theta,t) \, g(\nu_*) \dn_*\,. 
\end{equation}
\end{theorem}

The proof of Theorem \ref{thm} will be presented in Subsection \ref{proofthm}. 

 We now discuss the importance of the results. Equation (\ref{hl1}) 
is a continuity equation for the density $P_{\nu}$ of oscillators of a given natural frequency $\nu\,.$ Indeed, since the interactions do
not modify the natural frequencies ${\nu}$ of the oscillators, we must have an equation
representing the conservation of oscillators for each of these natural frequencies $\nu$. 
However, the external force under white
noise modifies the actual frequencies $w$ of the oscillators. This
interaction couples oscillators with different natural frequencies $\nu$. Therefore, the
mean frequency $V(\theta,t) = Y + KP$ is common to all oscillators (and consequently, does not
depend on $\nu$) and conforms to a balance equation which bears similarities with the gas
dynamics momentum conservation equations.  Having this information, 
we can recast the system of hydrodynamic limit equations \eqref{hl1}--\eqref{hl2}
in terms of the density function $P$ as follows. 

Multiplying \eqref{hl1} by $g(\nu)\,,$ the dependence on $\nu$ in the resulting equation can 
 be integrated out thanks to \eqref{P}, and applying \eqref{V} to \eqref{hl2}, we obtain the following system of equations:
 \begin{numcases}{}
\partial_t P + \partial_{\theta} (P(Y+KP))  = 0\,, \label{hl1b}\\
K P \, \partial_t P +  KP(Y+KP) \, \partial_{\theta} P + \partial_{\theta} P = 0\,. \label{hl2b}
\end{numcases}
The assumption \eqref{pt} means that $PY$ is a constant.  Hence, the second equation \eqref{hl2b} can be written in the conservation of momentum density
as follows:
\begin{equation}\label{cmoment}
    \frac{1}{2} \partial_t(PY + KP^2) + \frac{1}{2} \partial_{\theta}\left(KYP^2 + \frac{2}{3} K^2 P^3 + 2P\right) = 0.
\end{equation}
Therefore, $Y(\theta,t)$ and $P(\theta,t)$ (so $V(\theta,\nu,t)$) can first be computed by solving the system 
(\ref{hl1b})--(\ref{hl2b}).  Once $V(\theta,\nu,t)$ is found, equation (\ref{hl1}) is a 
transport equation with the provided coefficients, which can be simply integrated  
(given that the vector field $V$ is smooth, that is at least $C^1$).  Equation 
(\ref{hl1b}) represents the conservation of the total density $P$ of oscillators 
(that is, integrated with respect to $\nu \in \mathbb{R}$) with $V(\theta,t) = Y + KP\,.$ 



In preparation for the proof of our major Theorem \ref{thm}, we will show in the next Subsection a Hardy-type inequality.

\subsection{A Hardy-type inequality in \texorpdfstring{$\mathbb{R}$}{R}}\label{hard}
In order to state and prove a Hardy-type inequality over the whole real line, we first introduce our context, and more details of the preliminary results can be found in Appendices \ref{hardytype}, \ref{appii}, and \ref{hiw1}.  In particular, consider the space
\begin{equation}\label{w0}
W_0 = \left\{\varphi \in H^1(\mathbb{R}) \cap L^1(\mathbb{R}) \,, \int\limits_{\mathbb{R}} \varphi(w) \, \dw = 0\right\}\,.
\end{equation}
This space $W_0$ has some nice properties below.  
%
%

By \cite[Remark 13 on p.~217]{brezis}, $C_c^{\infty}(\mathbb{R})$ is dense in the Sobolev space $H^1(\mathbb{R})=W^{1,2}(\mathbb{R})\,,$ and thus $H^1(\mathbb{R}) = H_0^1(\mathbb{R})= W_0^{1,2}(\mathbb{R})= \overline{C_c^{\infty}(\mathbb{R})}^{\| \cdot \|_{H^1(\mathbb{R})}}\,,$ where $C_c^{\infty}(\mathbb{R})$ represents the space of infinitely differentiable functions with compact support in $\mathbb{R}$ \cite[p.~256, p.~259, and p.~273]{evans}. 

Let $\varphi \in W_0\,.$  Then, $\varphi \in H^1(\mathbb{R})\,,$ thus by \cite[Theorem 7.16]{leonipde} (as explained in Theorem \ref{mainH1AC} and Corollary \ref{corH1AC}), $\varphi$ admits an \textbf{absolutely continuous} representative $\overline{\varphi}: \mathbb{R} \to \mathbb{R}$ (belonging to $AC(\mathbb{R})$ defined in \eqref{ACrep}), which is also \textbf{locally absolutely continuous} (that is, $\overline{\varphi}$ is in $AC_{\tu{loc}}(\mathbb{R})$ defined by \eqref{aclocr}), with $\varphi=\overline{\varphi}$ almost everywhere (a.e.), such that both $\overline{\varphi}$ and its classical derivative $\overline{\varphi}'$ belong to $L^2(\mathbb{R})\,.$
Moreover, by Corollary \ref{corH1AC}, $\overline{\varphi}$ is H\"{o}lder continuous of exponent $1/2\,,$ uniformly continuous, and continuous over all $\mathbb{R}\,,$ so $\overline{\varphi} \in C(\mathbb{R})\,.$  By Remark \ref{idh1ac}, throughout this paper, \textbf{an element $\varphi \in H^1(\mathbb{R}) \subset H^1_{\tu{loc}}(\mathbb{R})$ is assumed to identically be its locally absolutely continuous representative $\overline{\varphi}$ in $AC_{\tu{loc}}(\mathbb{R})$} (defined by \eqref{aclocr}).  


Since $\varphi \in C(\mathbb{R})$ and $\varphi \in L^1(\mathbb{R})$ (as $\varphi \in W_0$), it follows from Lemma \ref{irl} that the Lebesgue integral equals to the improper Riemann integral over all $\mathbb{R}$:
\[\int_{\mathbb{R}} \varphi \, \dm = \int_{-\infty}^{\infty}  \varphi (w)  \, \dw\,.\]

\bigskip
Now, we present our Hardy-type inequality developed from the Hardy's inequality in \cite[Theorem~1.14]{hardyopic} or \cite[Theorem 4]{Muck1972} by Muckenhoupt (explained in Theorem \ref{ttam} within our context), over the domain of the whole real line. 

By Proposition \ref{muac}, the function $M_V(w) = \dfrac{1}{\sqrt{2\pi}} \tu{exp} \left( -\dfrac{1}{2}(w-V)^2 \right)$ defined in \eqref{equi} is in $AC_{\tu{loc}}(\mathbb{R})$ \eqref{aclocr}, thus $M_V(w)$ is continuous (so measurable),
positive, and almost everywhere (a.e.) finite on $(-\infty,\infty)\,;$ that is, $M_V(w)$ is a continuous weight function over $\mathbb{R}$ (by Definition \ref{weightdef}).

Also by Proposition \ref{muac}, for any $\varphi \in W_0$ defined in \eqref{w0}, our function $\dd u(w) = \frac{\varphi(w)}{M_V(w)}$ (specified by \eqref{upm}) belongs to $AC_{\tu{loc}}(\mathbb{R})$ \eqref{aclocr}.  In these settings, we obtain the following result.

\begin{theorem}[\textbf{Hardy-type inequality in $\mathbb{R}$}]\label{hardylem1}
There exists a finite constant $\hat{C}$ such that for any $u(w) \in AC_{\tu{loc}}(\mathbb{R})\,,$ the following inequality holds:
\begin{equation}\label{hardyineq}
\int_{-\infty}^{\infty}(u(w))^2 M_V(w) \, \dw  \leq \hat{C} \int_{-\infty}^{\infty} (u'(w))^2 M_V(w) \, \dw \,.
\end{equation}
\end{theorem}

\begin{proof}
By Proposition \ref{d0}, there is some $d\in \mathbb{R}$ such that $\varphi(d) = 0\,,$ so $u(d) = 0\,.$  We then consider two cases of subdomains of $w\,:$ Case 1 ($w \in [d,\infty)$) and Case 2 ($w \in (-\infty,d]$).   

\noindent
\textbf{Case 1.}  We prove the Hardy-type inequality with $w$ in $[d,\infty)\,,$ that is, for all $u \in AC_{\tu{L}}([d,\infty))\,,$ there exists a finite constant $C_{\tu{L}} > 0$ such that
\begin{equation}\label{hardyineq1}
\int_{d}^{\infty}(u(w))^2 M_V(w) \, \dw  \leq C_{\tu{L}} \int_{d}^{\infty} (u'(w))^2 M_V(w) \, \dw \,.
\end{equation}

By Theorem \ref{ttam}, for the validity of \eqref{hardyineq1}, we only need to show that
\begin{equation}\label{BL1}
B_{\tu{L}} = \sup_{d < r < \infty} \left( \int_r^{\infty} M_V(w) \, \dw \right) \left( \int_d^r \frac{1}{M_V(w)} \, \dw \right) < \infty\,.
\end{equation}
(Within \eqref{BL1}, it is known that $1/M_V(w) = +\infty$ if $M_V(w) = 0\,;$  for the product inside the supremum, $0\cdot \infty$ implies $0\,.$)
%

To prove \eqref{BL1}, we first let $x=\dfrac{w-V}{\sqrt{2}}$ to get the following explicit expressions:
\begin{align}\label{finiteBL1}
\begin{split}
\int_r^{\infty} M_V(w) \, \dw = \int_r^{\infty} \frac{1}{\sqrt{2\pi}} \,  \tu{exp} \left( -\frac{1}{2}(w-V)^2 \right) \, \dw = \frac{1}{\sqrt{2\pi}}  \sqrt{2} \int_{\frac{r-V}{\sqrt{2}}}^{\infty} \tu{exp} \left( -x^2 \right) \, \dx\,, 
\end{split}
\end{align}
and
\begin{align}\label{finiteBL2}
\begin{split}
\int_d^{r} \frac{1}{M_V(w)} \, \dw = \int_d^{r} \sqrt{2\pi} \, \tu{exp} \left( \frac{1}{2}(w-V)^2 \right) \, \dw =
\sqrt{2\pi} \sqrt{2} \int_{\frac{d-V}{\sqrt{2}}}^{\frac{r-V}{\sqrt{2}}} \tu{exp} \left(x^2 \right) \, \dx\,.
\end{split}
\end{align}
Without loss of generality, assume $0 \leq d-V < w-V < \infty$ and $0 \leq d - V < r -V <\infty\,.$  We then denote 
\[x= \frac{w-V}{\sqrt{2}}\,, \quad a= \frac{r-V}{\sqrt{2}}\,, \quad b= \frac{d-V}{\sqrt{2}}\]
and thus assume
\[0 \leq b < x < \infty\,, \quad 0 \leq b < a < \infty\,.\]

Now, we verify that $B_{\tu{L}}$ in \eqref{BL1} is finite.  Indeed, by using optimal truncation for the asymptotic expansion as $r \to \infty$ or $a \to \infty$ (with integration by parts), we obtain
%
%
%
%
%
%
%
%
\begin{equation}\label{hconst}
\int_a^{\infty} e^{-x^2} \, \dx \approx  \frac{1}{2a} e^{-a^2}\,, \quad \int_b^a e^{x^2} \, \dx \approx  \frac{1}{2a} e^{a^2}\,.
\end{equation}
%
%
More specifically, for the first integral of \eqref{hconst}, using the change of variables $u = x^2\,,$ we get
\begin{equation}\label{hconst1}
\int_a^{\infty} e^{-x^2} \, \dx  = \frac{1}{2} \int_{a^2}^{\infty} \frac{e^{-u}}{\sqrt{u}} \, \du = \frac{1}{2} \left( \frac{e^{-a^2}}{\sqrt{a^2}} -\frac{1}{2} \int_{a^2}^{\infty} \frac{e^{-u}}{u^{3/2}} \, \du \right)\,,
\end{equation}
where the last equality is obtained by integrating by parts. 
We can continue by employing partial integration again and over until we reach an asymptotic series:
\begin{equation}\label{hconst1s}
\int_a^{\infty} e^{-x^2} \, \dx   =e^{-a^2} \left( \frac{1}{2a} - \frac{1}{4a^3} + \frac{3}{8a^5} - \frac{15}{16a^7} + \mathcal{O}\left(  \left( \frac{1}{a} \right)^8 \right) \right) \,.
\end{equation}
This is an asymptotic expansion (series) for $\dd \int_a^{\infty} e^{-x^2} \, \dx$ as $a \to \infty\,.$
(Note that $e^{-1/w^2}$ cannot be represented by a Taylor series around $w_0=0\,,$ but $e^{-1/w^2}$ can be expressed as a Laurent series by replacing $w$ with $-1/w^2$ in the power series for the exponential function $e^w\,.$)
%
%
Similarly, the second integral of \eqref{hconst} (after the change of variable $u=x^2$) has the expression
\begin{align}\label{hconst2}
\begin{split}
\int_b^{a} e^{x^2} \, \dx  &= \frac{1}{2} \int_{b^2}^{a^2} \frac{e^{u}}{\sqrt{u}} \, \du = \frac{1}{2} \left( \frac{e^{a^2}}{\sqrt{a^2}} +\frac{1}{2} \int_{b^2}^{a^2} \frac{e^{u}}{u^{3/2}} \, \du \right)\\ 
&=e^{a^2} \left( \frac{1}{2a} + \frac{1}{4a^3} + \frac{3}{8a^5} + \frac{15}{16a^7} + \mathcal{O}\left(  \left( \frac{1}{a} \right)^8 \right) \right)\,.
\end{split}
\end{align}
Thus, 
\begin{equation}\label{BLapp1}
B_{\tu{L}} \approx \frac{1}{4a^2} = \frac{1}{4\left(\dfrac{r-V}{\sqrt{2}}\right)^2} < \infty\,.
\end{equation}
Hence, \eqref{BL1} is true. 

Recall that by Proposition \ref{muac}, $M_V(w) \in AC_{\tu{loc}}(\mathbb{R})\,,$ so $M_V(w)$ is continuous weight function in $\mathbb{R}\,.$ 

Therefore, applying \cite[Theorem~1.14]{hardyopic} or \cite[Theorem 4]{Muck1972} by Muckenhoupt (restated in Theorem \ref{ttam} for our settings), since \eqref{BL1} is valid, it follows that the Hardy-type inequality \eqref{hardyineq1} holds; furthermore, if $C_{\tu{L}}$ is the smallest constant for which
\eqref{hardyineq1} is satisfied, then
\begin{equation}\label{blcl1}
B_{\tu{L}} \leq C_{\tu{L}} \leq 4B_{\tu{L}}\,.
\end{equation}

\bigskip
\noindent
\textbf{Case 2.}  We prove the Hardy-type inequality on $(-\infty,d]\,,$ that is, for all $u$ in $AC_R((-\infty,d])\,,$ there exists a finite constant $\tilde{C}_{\tu{L}}$ such that
\begin{equation}\label{hardyineq2}
\int_{-\infty}^d (u(w))^2 M_V(w) \, \dw \leq \tilde{C}_{\tu{L}}  \int_{-\infty}^d (u'(w))^2 M_V(w) \, \dw \,.
\end{equation}
Benefiting from the approach in \cite[p.~9]{hardyopic}, our substitution $y =-w$ transforms $u(w) \in AC_{\tu{R}}((-\infty,d])$ into $\tilde{u}(y) \in AC_{\tu{L}}([-d,\infty))\,,$ with $\tilde{u}(y) = u(-y)\,,$ thus $ \tilde{u}(-d) = u(d)= 0\,,$ and with $\tilde{M}_V(y) = M_V(-y)\,.$ 
The inequality \eqref{hardyineq2} is equivalent to
\begin{equation}\label{hardyineq2i}
\int_{-d}^{\infty} (\tilde{u}(y))^2 \tilde{M}_V(y) \, \dy  \leq \tilde{C}_{\tu{L}} \int_{-d}^{\infty} (\tilde{u}'(y))^2 \tilde{M}_V(y) \, \dy \,.
\end{equation}
%
Hence, instead of proving \eqref{hardyineq2}, we show that \eqref{hardyineq2i} holds.
Indeed, by a similar argument to Case 1, it is true that \begin{equation}\label{BL2}
\tilde{B}_{\tu{L}} = \sup_{-d < r < \infty} \left( \int_r^{\infty} M_V(w) \, \dw \right) \left( \int_{-d}^r \frac{1}{M_V(w)} \, \dw \right) < \infty\,.
\end{equation}

Again, using \cite[Theorem~1.14]{hardyopic} or \cite[Theorem 4]{Muck1972} by Muckenhoupt (described by Theorem \ref{ttam} within our context) with $\tilde{u}(-d)=0\,,$ since \eqref{BL2} is valid, it follows that the Hardy-type inequality \eqref{hardyineq2i} (equivalently, \eqref{hardyineq2}) holds; furthermore, if $\tilde{C}_{\tu{L}}$ is the smallest constant for which
\eqref{hardyineq2} is satisfied, then
\begin{equation}\label{blcl2}
\tilde{B}_{\tu{L}} \leq \tilde{C}_{\tu{L}} \leq 4\tilde{B}_{\tu{L}}\,.
\end{equation}

\bigskip

We choose $\hat{C}$ such that
\begin{equation}\label{hatCL}
\tu{max}\{B_{\tu{L}},\tilde{B}_{\tu{L}}\} \leq \hat{C} \leq \tu{min} \{4B_L,4\tilde{B}_{\tu{L}}\}\,.
\end{equation}
Then, \eqref{hardyineq1} and \eqref{hardyineq2} (by the non-negativity of the integrands) respectively hold with the selected $\hat{C}$ from \eqref{hatCL}:
\begin{equation}\label{hardyhat1}
\int_d^{\infty}(u(w))^2 M_V(w) \, \dw  \leq \hat{C} \int_d^{\infty} (u'(w))^2 M_V(w) \, \dw 
\end{equation}
and
\begin{equation}\label{hardyhat2}
\int_{-\infty}^d(u(w))^2 M_V(w) \, \dw  \leq \hat{C} \int_{-\infty}^d (u'(w))^2 M_V(w) \, \dw\,. 
\end{equation}
Therefore, summing \eqref{hardyhat1} and \eqref{hardyhat2}, with $M_V(w)$ defined by \eqref{equi}, $\varphi \in W_0$ specified in \eqref{w0}, and $u(w)$ given by \eqref{upm} (so $u(w) \in AC_{\tu{loc}}(\mathbb{R})$ \eqref{aclocr}), we obtain the desired inequality
\begin{equation}\label{hardyhat}
\int_{-\infty}^{+\infty}(u(w))^2 M_V(w) \, \dw  \leq \hat{C} \int_{-\infty}^{+\infty} (u'(w))^2 M_V(w) \, \dw\,. 
\end{equation}
\end{proof}

In order to prove the main Theorem \ref{thm}, we also need the following Subsection, thanks to \cite{PD1} and the references therein.

\subsection{Properties of collision operator Q}\label{qprop}
Given $\nu \in \text{supp}\, g(\nu)$, using $\Phi_f$ in \eqref{K}, we let
\begin{equation}\label{uf}
V[f]=V[f](\theta,\nu,t) = \nu + K\Phi_f (\theta,t) \,.
\end{equation}
Note that the (normalized) local Maxwellian of (\ref{per2}) is of the form
\begin{align}\label{Maxnew}
M_{V[f^{\e}]}(\theta, w,\nu, t) &= \frac{1}{\sqrt{2\pi}} \text{ exp}
\left(-\frac{1}{2} \left(w -  V[f^{\e}](\theta, \nu,t)\right)^2\right) \,.
\end{align}

Provided $V \in \mathbb{R}\,,$
we define the linear operator
\begin{equation}\label{Qcal}
 \mathcal{Q}_V (f)(w,\nu) := \partial_w \left(M_V(w) \, \partial_w 
\left(\frac{f(w,\nu)}{M_V(w)}\right)\right)\,.
\end{equation}

\begin{remark}\label{equildef}
Now, thanks to (\ref{per2}), we have $Q(f^{\e}) = \mathcal{O}(\e)$.  Letting $\e \to 0$ in \eqref{per2} induces 
$Q(f^0) = 0$.  Thus, $f^0$ is a so-called equilibrium, that is, a solution of $Q(f) = 0\,.$
\end{remark}

Since 
$Q$ only operates on the $(w, \nu)$ variables, we first ignore the phase-temporal dependence.  The following result then holds.

\begin{lemma}\label{qq}
 The operator $Q(f(w,\nu))$ can be written as
\begin{equation}\label{qre}
  Q(f) = \mathcal{Q}_{V[f]}(f)\,.
 \end{equation}
\end{lemma}

\begin{proof}
It is clear that
\begin{align*}
 \mathcal{Q}_{V[f]}(f) & = 
 \partial_w \left(M_{V[f]} \, \partial_w \left(\frac{f}{M_{V[f]}} 
 \right)\right)\\
 &=\partial_w \left(M_{V[f]} \,  \frac{(\partial_w f) \, M_{V[f]} + (w-V[f]) \, f \, M_{V[f]}}
{M^2_{V[f]}}\right)\\
 & = \partial_w ( \partial_w f + (w - V[f]) f)\\
 & = Q(f)\,.
\end{align*}
\end{proof}

In the rest of this paper, we need the definition of Bochner space as follows.  Given a measure space $(Y, \Sigma, \mu)$ and a Banach space $(X, \|\cdot\|_X)\,,$ for every $1 \leq r < 
\infty\,,$ we use $L^r(Y;X)$ to represent the Bochner space \cite{evans,papa_bochner}
with the norms 
\[\|\phi\|_{L^r(Y;X)} := \left(\int_Y \|\phi(y) \|_{X}^r \, \dy\right)^{1/r} < + \infty\,, 
\]
\[ \|\phi\|_{L^{\infty}(Y;X)}: = \tu{ess sup}_{y \in Y} \|\phi(y)\|_{X}  < + \infty\,,\]
where $(X, \| \cdot \|_{X})$ is a Banach space, for example $X=H^1(\mathbb{R}) \,.$ 

Using the above definition of Bochner space, we now introduce the functional setting as in \cite{PD1}.  Let $f(w,\nu)$ and $h(w,\nu)$ be smooth functions of $(w,\nu)$ 
with fast decay when $w \to \pm \infty$.  The duality products are defined as follows:
\begin{align}\label{dualprod}
\begin{split}
&\langle f, h \rangle_{0,V} : = \iint\limits_{\mathbb{R} \; \mathbb{R}} 
f(w,\nu) \, h(w,\nu) \, \frac{1}{M_V (w)} \, g(\nu) \, \dw \, \dn \,,\\
&\langle f, h \rangle_{1,V} := \iint\limits_{\mathbb{R} \; \mathbb{R}} 
\partial_w\left(\frac{f(w,\nu)}{M_V(w)}\right) 
\partial_w \left(\frac{h(w,\nu)}{M_V(w)}\right) M_V (w) \, g(\nu) \, \dw \, \dn \,.
\end{split}
\end{align}

\noindent Then, $\langle f, h \rangle_{0,V}$ defines a duality (that is, a continuous bilinear form) between 
$f \in L^1(\mathbb{R}; L^2(\mathbb{R}))$ and $h \in L^{\infty}(\mathbb{R};L^2(\mathbb{R}))$.  
Similarly, $\langle f, h \rangle_{1,V}$ defines a duality between 
$f \in L^1(\mathbb{R}; H^1(\mathbb{R}))$ 
and $h \in L^{\infty}(\mathbb{R}; H^1(\mathbb{R}))$.  Applying the Green's formula to 
smooth functions, we have
\begin{equation}\label{d1}
 -\langle \mathcal{Q}_V (f), h \rangle_{0,V} = \langle f, h \rangle_{1,V}\,.
\end{equation}
Thus, for $f(w,\nu) \in L^{\infty}(\mathbb{R}; L^2(\mathbb{R}))$, we define $\mathcal{Q}_V (f)$ as 
a linear form on $L^{\infty}(\mathbb{R}; L^2(\mathbb{R}))$.  
Now, the collection of equilibria is described as follows.

\begin{definition}[\cite{PD1}]\label{m1}
 The set $\mathcal{E}$ of equilibria of $Q$ is given by
\begin{equation}\label{equilE}
\mathcal{E} = \{f(w,\nu) \in L^1(\mathbb{R}; H^1(\mathbb{R})) \: | \: f \geq 0 \text{ and } 
 \mathcal{Q}_{V[f]}(f) = 0\}\,.
 \end{equation}
\end{definition}

\vspace{12pt}

The following Lemma characterizes $\mathcal{E}$.
\begin{lemma}\label{m2}
 The set $\mathcal{E}$ of equilibria is the set of all functions of the form
\begin{equation}\label{m3}
  w \mapsto P_{\nu} \, M_V (w)\,,
 \end{equation}
 where the function $\nu \in \mathbb{R}
 \mapsto P_{\nu} \in \mathbb{R}_{+}$ is arbitrary in 
 the set $L^1(\mathbb{R})$ and meets the conditions (\ref{Pcond}) as well as (\ref{pt}), $\dd V= \nu + K\int\limits_{\mathbb{R}}P_{\nu_*}  \, g(\nu_*) \, \dn_*$ in $\mathbb{R}\,,$ and the normalization (\ref{normM}) of $M_V(w)$ is satisfied.
\end{lemma}

\begin{proof}
First, suppose that $f(w,\nu)$ is an equilibrium, that is, $f \in \mathcal{E}$ in the Definition \ref{m1}.  Then, with $M_{V[f]}$ defined in \eqref{Maxnew}, 
 we obtain from (\ref{d1}) that 
 \begin{align*}
  0 &= - \langle \mathcal{Q}_{V[f]} (f), f \rangle_{0, V[f]} \\
  &= - \iint\limits_{\mathbb{R} \; \mathbb{R}}\left(\partial_w\left(M_{V[f]} \,
\partial_w\left(\frac{f}{M_{V[f]}}\right)\right)\right) \, f \,  \frac{1}{M_{V[f]}} \, g(\nu) \, \dw \, \dn\\
  & = \iint\limits_{\mathbb{R} \; \mathbb{R}} M_{V[f]} \, 
\partial_w\left(\frac{f}{M_{V[f]}}\right) \partial_w 
  \left(\frac{f}{M_{V[f]}}\right)   g(\nu) \, \dw \, \dn \\
&=\langle f, f \rangle_{1, V[f]}\,.
\end{align*}

Since $M_{V[f]}$ is bounded from below and above, it follows that 
\[\partial_w\left(\frac{f(w, \nu)}{M_{V[f]}}\right) = 0\,,\]
that is, there exists $P_{\nu} \in \mathbb{R}$, independent of $w$, such that 
$f(w, \nu) = P_{\nu} M_{V[f]}(w)\,.$ In addition, as $f \in L^1(\mathbb{R}; 
H^1(\mathbb{R}))$ and $f \geq 0$, it follows that $P_{\nu} \geq 0\,,$ and that the function 
$\nu \in \mathbb{R} \mapsto P_{\nu} \in \mathbb{R}_{+}$ belongs to $L^1(\mathbb{R})$.  
Thus, $f(w,\nu)$ is of the form (\ref{m3}).

Conversely, suppose that $\rho(w,\nu)=P_{\nu} \, M_V (w)$ is of the form (\ref{m3}) with 
$P_{\nu} \in \mathbb{R}_{+}$ being as regular as in the 
Lemma's hypothesis.  We first need to show that $V[\rho] \text{ } (= V[P_{\nu} M_V]) = V\,.$ Indeed, thanks to (\ref{uf}), \eqref{K}, \eqref{normM}, \eqref{P}, and (\ref{V}), 
respectively, we have
\begin{align}\label{Vequil}
\begin{split}
V[P_{\nu} M_V]&= \nu + 
K \, \Phi_{P_{\nu} M_V} (\theta,t)\\
&=  \nu + K \iint\limits_{\mathbb{R} \; \mathbb{R}} M_V \, P_{\nu_*} \, g(\nu_*) \, \dw \, \dn_*\\
&=  \nu + K  \int\limits_{\mathbb{R}} \left(\int\limits_{\mathbb{R}} M_V \,  \dw \right) 
P_{\nu_*} \, g(\nu_*) \, \dn_* \\
 &=  \nu + K   \int\limits_{\mathbb{R}}  P_{\nu_*} \, g(\nu_*) \, \dn_*\\
 & = \nu + K \, P\\
 &= V\,.
 \end{split}
\end{align}
By \eqref{Qcheck} below, $Q(\rho(w,\nu)) = Q(P_{\nu} \, M_V(w))= 0\,.$ Thus, $\rho(w,\nu) = P_{\nu} \, M_V (w)$ is obviously an equilibrium by 
\eqref{qre} and \eqref{equilE} of Definition \ref{m1}.  
\end{proof}

\vspace{10pt}
From Lemma \ref{m2} and the fact that $f^0$ is an equilibrium by Remark \ref{equildef}, we derive 
\[f^0(\theta,w, \nu,t) = P_{\nu}(\theta, t) M_{V(\theta,\nu,t)}(w)\,,\]
as in (\ref{hl}).  
%
Using the expression \eqref{Q}, with \eqref{mathcalF}, \eqref{Vequil}, and \eqref{V}, it can be checked that $Q(f^0) = 0$ as follows:
\begin{align}\label{Qcheck}
\begin{split}
 Q(f^0) & = \partial_w \left(\mathcal{F}[f^0] \, f^0 \right) + (\partial_w^2 f^0)\\
 &= \partial_w \left(\mathcal{F}[P_{\nu}  M_V] \, P_{\nu} M_V\right) + 
 (\partial_w^2 (P_{\nu} M_V)) \\
& = \partial_w\left(\left( w - V \right)
P_{\nu} M_V\right) + 
 \partial_w^2 (P_{\nu} M_V)\\
 &= (P_{\nu} M_V + P_{\nu}(w -V) \, \partial_w M_V) + 
 (\partial_w ( -P_{\nu} M_V(w -V)))\\
 &= (P_{\nu} M_V - (w - V)^2 P_{\nu} M_V ) + 
 (-P_{\nu} M_V + (w -V)^2 P_{\nu} M_V) = 0\,.
 \end{split}
\end{align}
Now, because $Q$ acts on the $(w, \nu)$ variables only, $P_{\nu} = P_{\nu} (\theta,t)$ and $V = V(\theta,\nu,t)$ are a priori 
arbitrary functions of $(\theta,t)\,,$ except in (\ref{Pcond}) that $\dd \int\limits_{\mathbb{T}} P_{\nu}(\theta,t) \, \dth = 1\,.$   
Thus, the fact from \eqref{Qcheck} that $Q(f^0(\theta,w, \nu,t)) = 0$ does not impose any condition on the dependence 
of $f^0$ on $(\theta,t)$.  In order to determine how $P_{\nu} (\theta,t)$ and $V(\theta,\nu,t)$ depend 
on $(\theta,t)$, we need the next Subsection as in \cite{PD1}.

\subsection{Generalized Collision Invariants (GCI)}\label{gci0}

First, we should get to know the notion of Collision Invariant as follows.
\begin{definition}\label{ci}
 A collision invariant (CI) is a function $\psi(w,\nu) \in L^{\infty}
 (\mathbb{R}; H^1(\mathbb{R}))$ such that for all functions $f(w,\nu) \in L^1(\mathbb{R}; 
 H^1(\mathbb{R}))$, we have
 \begin{equation}\label{cidf}
  0= - \iint\limits_{\mathbb{R} \; \mathbb{R}} Q(f) \, \psi  \, g(\nu) \, \dw \, \dn 
  : =
  \langle \psi M_{V[f]},f \rangle_{1,V[f]} \,.
 \end{equation}
 These Collision Invariants form a set $\mathcal{C}$, as a 
 vector space.
\end{definition}

Note that the relation (\ref{cidf}) is obtained via Lemma \ref{qq} and relation (\ref{d1}), 
respectively:
\begin{align*}
 - \iint\limits_{\mathbb{R} \; \mathbb{R}} Q(f) \, \psi \,  g(\nu) \, \dw \, \dn 
 & = - \iint\limits_{\mathbb{R} \; \mathbb{R}} (\mathcal{Q}_{V[f]}(f)) 
 (\psi M_{V[f]}) \, \frac{1}{M_{V[f]}} \, g(\nu) \, \dw \, \dn \\
 & = - \langle \mathcal{Q}_{V[f]} (f) , \psi M_{V[f]}\rangle_{0,V[f]}\\
 & = \langle \psi M_{V[f]},f \rangle_{1,V[f]}\,.
\end{align*}

We then deduce the following result.
\begin{proposition}[\cite{PD1}]
 Any function $\phi: \nu \in \mathbb{R} \mapsto \phi(\nu) \in \mathbb{R}$ 
 belonging to $L^{\infty}(\mathbb{R})$ is a CI.
\end{proposition}

\begin{proof}
 Let $\phi \in L^{\infty}(\mathbb{R})$ and $f(w,\nu) \in L^1(\mathbb{R}; H^1(\mathbb{R}))$.
 Since $M_{V[f]}$ is bounded, it readily follows that $\phi M_{V[f]} \in 
 L^{\infty}(\mathbb{R}; H^1(\mathbb{R}))$.  Now, as $\phi$ does not depend on $w$, it 
 satisfies (\ref{cidf}).
\end{proof}

Later, we will see that this set of CI does not suffice to provide the phase-temporal 
evolution of $P_{\nu}$ and $V$ in the hydrodynamic limit.  
While other apparent Collision Invariants are also absent, we benefit from a weaker concept, 
that of ``Generalized Collision Invariant'' (GCI) from \cite{gci1, hlk2, PD1} below. 

\begin{definition}\label{gci1}
 
 
Let $V \in \mathbb{R}$ be given.  A Generalized Collision Invariant (GCI) associated to $V$ is 
 a function $\psi(w,\nu) \in L^{\infty}(\mathbb{R}; H^1(\mathbb{R}))$ such that the following 
 property holds:  For all functions $f(w, \nu)$ satisfying $f \in L^1(\mathbb{R}; H^1(\mathbb{R}))$ as well as  
\begin{equation}\label{hypo}
 V[f] =V \quad \text{and} \quad \iint\limits_{\mathbb{R} \; \mathbb{R}} (w - V) \, f  \, g(\nu) \, \dw \, \dn = 0 \,,
 \end{equation}
we obtain
\begin{equation}\label{gci2}
  0 = - \iint\limits_{\mathbb{R} \; \mathbb{R}} Q(f) \, \psi  \, g(\nu) \, \dw \, \dn 
  = - \iint\limits_{\mathbb{R} \; \mathbb{R}} \mathcal{Q}_V(f) \, \psi \, g(\nu) \, \dw \, \dn \\
  := \langle \psi M_V,f\rangle_{1,V} \,.
 \end{equation}
 These Generalized Collision Invariants form a set $\mathcal{G}_V$, as a vector space.
\end{definition}

Similarly to the case involving the CI, if $\psi(w,\nu)$ is in 
$L^{\infty}(\mathbb{R}; H^1(\mathbb{R}))$, 
then $\psi M_V(w)$ is too, and (\ref{gci2}) is well-defined.  In order to identify $\mathcal{G}_V$, 
we introduce a proper functional setting for functions of $w$ only \cite{PD1}, 
as follows.  

Consider the space $W_0$ specified in \eqref{w0}. 
 Given $V \in \mathbb{R}\,,$ for $f\in W_0\,,$ we define the following norms or semi-norms on $L^2(\mathbb{R})$ 
and $H^1(\mathbb{R})$:
\begin{align}\label{seminorm}
|f|_{0,V}^2 : = \int\limits_{\mathbb{R}} |f(w)|^2 \, \frac{1}{M_V(w)} \, \dw\,, \quad
|f|^2_{1,V} : = \int\limits_{\mathbb{R}} \left| \partial_w \left(
\frac{f(w)}{M_V(w)}\right) \right|^2 \, M_V(w) \, \dw \,.
\end{align}

\noindent These two semi-norms are respectively equivalent to the classical $L^2$ norm and $H^1$ 
semi-norm on $L^2(\mathbb{R})$ and $H^1(\mathbb{R})$.  Using the Hardy-type inequality \eqref{hardyineq}, we have 
\begin{equation}\label{pieq}
 |\varphi|^2_{0,V} \leq  C |\varphi|^2_{1,V} \,, \quad \forall \varphi \in W_0\,, 
\end{equation}
where $C$ is a positive constant.  The associated bilinear forms are denoted by 
\begin{align}\label{biforms}
\begin{split}
&(f, h)_{0,V} : = \int\limits_{\mathbb{R} \; \mathbb{R}} 
f(w) \, h(w) \, \frac{1}{M_V (w)} \, \dw  \,,\\
&(f, h)_{1,V} := \int\limits_{\mathbb{R} \; \mathbb{R}} 
\partial_w\left(\frac{f(w)}{M_V(w)}\right) 
\partial_w \left(\frac{h(w)}{M_V(w)}\right) M_V (w) \, \dw  \,.
\end{split}
\end{align}
Note that in these forms, the functions $f$ and $h$ are over $w \in \mathbb{R}\,.$  Whereas, in the duality products \eqref{dualprod}, the functions $f$ and $h$ are over $(w,\nu) \in \mathbb{R} \times \mathbb{R}\,.$

\bigskip
We are now ready to characterize $\mathcal{G}_V$ in the following Proposition, 
as in \cite{PD1}.

\begin{proposition}\label{gcis}
 The set of GCI is of the form
\begin{equation}\label{gcis2}
\mathcal{G}_V = \{\psi(w,\nu) =\beta \, \chi_V(w) + \phi(\nu), \quad \beta \in \mathbb{R}, 
  \quad \phi(\nu) \in L^{\infty}(\mathbb{R})\}\,,
 \end{equation}
 where $\varphi_V (w) = \chi_V(w) \, M_V(w)$ is the unique solution in $W_0$ of the following variational problem:
\begin{equation}\label{vf}
 \text{Find } \varphi(w) \in W_0 \text{ such that } (\varphi,f)_{1,V} = 
 ((w-V) M_V, f)_{0,V}\,,
 \quad \forall f(w) \in H^1(\mathbb{R}) \,.
 \end{equation}
\end{proposition}

\begin{proof}
 The existence of a unique solution $\varphi_V(w) \in W_0$ for the variational problem 
 (\ref{vf}) is based on the Lax--Milgram theorem and the Hardy-type inequality (\ref{pieq}) acting as the Poincar\'{e} inequality; 
 see \cite{PD1,hlk2,expand1}, for more details.

Now, given $\psi(w,\nu) \in \mathcal{G}_V$ in Definition \ref{gci1}, and 
$f(w,\nu)  \in L^1(\mathbb{R}; H^1(\mathbb{R}))\,.$  We note first that from Definition \ref{gci1}, using the duality products \eqref{dualprod}, the condition (\ref{hypo}) is equivalent to 
\[\left\langle (w-V)M_V, f \right\rangle_{0,V} = 0\,.\]
Then, by (\ref{gci2}), 
$\psi(w,\nu)$ is a GCI if and only if $\psi(w,\nu) \in L^{\infty}(\mathbb{R}; H^1(\mathbb{R}))$ and 
the following implication holds:  for all $f(w,\nu) \in L^1(\mathbb{R}; H^1(\mathbb{R}))$,
\[\left\langle (w-V)M_V, f \right\rangle_{0,V} = 0 
\implies \langle \psi M_V, f\rangle_{1,V} = 0\,.\]

It means that (by a typical functional analytical argument) there exists 
a real number $\beta$ such that
\begin{equation}\label{vp2}
 \langle \psi M_V, f\rangle_{1,V}= \beta 
 \left\langle (w-V)M_V, f \right\rangle_{0,V}  \,, 
 \quad \forall f(w,\nu) \in L^1(\mathbb{R}; H^1(\mathbb{R}))\,.
\end{equation}

\noindent Thus, $\psi(w,\nu)$ is the solution of an elliptic variational problem.

Provided the unique $\chi_V (w)$ found from (\ref{vf}) and for any
 $\phi(\nu) \in L^{\infty}(\mathbb{R})$, we notice that the function $\psi(w,\nu)$ with 
 $(w,\nu) \mapsto \beta\chi_V (w) + \phi(\nu)$ 
 belongs to $L^{\infty}(\mathbb{R}; H^1(\mathbb{R}))$ 
and satisfies the variational problem (\ref{vp2}).  This function 
$\psi(w,\nu) = \beta \chi_V (w) + \phi(\nu)$ is unique.  Indeed, 
by linearity, the difference $\psi(w,\nu)$ of two such solutions ($\psi_1(w,\nu)$ and $\psi_2(w,\nu)$) lies in 
$L^{\infty}(\mathbb{R}; H^1(\mathbb{R}))$ and satisfies
\[\langle \psi M_V, f\rangle_{1,V}=0\,, 
\quad \forall f(w,\nu) \in L^1(\mathbb{R}; H^1(\mathbb{R}))\,.\]
 
\noindent Let $\zeta_a(\nu)$ be the indicator function of the interval 
$\mathbb{I} = [-a,a]\,,$ with $a >0\,,$ and choosing $f= \psi M_V \zeta_a$ as a test function in 
$L^1(\mathbb{R}; H^1(\mathbb{R}))$, we obtain
\begin{align*}
 0 &= \langle \psi M_V, \psi M_V \zeta_a\rangle_{1,V}\\
 & = \iint\limits_{\mathbb{R} \; \mathbb{R}} \partial_w 
 \left(\frac{\psi M_V}{M_V}\right) \, 
 \partial_w \left(\frac{\psi M_V \zeta_a}{M_V}\right) \, M_V(w) \,  g(\nu) \, \dw \, \dn \\
 & = \iint\limits_{\mathbb{I} \; \mathbb{R}} |\partial_w \psi|^2 \, M_V(w) \, g(\nu) \, \dw \, \dn \,.
\end{align*}
This means that the difference $\psi=\psi_1 - \psi_2$ does not depend on $w$ and hence of the form 
$\psi(\nu)$, with $\psi(\nu) \in L^{\infty}(\mathbb{R})$.  
Therefore, the proof is completed.
\end{proof}

The distributional sense of the variational problem (\ref{vf}) is that $\chi_V(w)$ is a solution of the following elliptic problem:
\begin{equation}\label{vf2}
 -\partial_w(M_V(w) \, (\partial_w \chi_V(w))) =  (w-V)M_V(w) \,,
 \quad \int\limits_{\mathbb{R}} 
 \chi_V(w) \, M_V(w) \, \dw = 0\,,
\end{equation}
where the second equation comes from (\ref{w0}).

\noindent Then, using Lax--Milgram's Theorem as in \cite{PD1,hlk2,expand1,chi1}, 
we obtain a unique solution of 
(\ref{vf2}):
\begin{equation}\label{sol}
 \chi_V(w) = w - V \,.
\end{equation}


\subsection{Hydrodynamic limit}\label{proofthm}

In this subsection, we focus on the proof of Theorem \ref{thm}, when taking the limit $\e \to 0$, and thanks to \cite{PD1}.

\begin{proof}[\textbf{Proof of Theorem \ref{thm}}.] 
From Lemma \ref{m2} and the fact that $\dd f^0 = \lim_{\e \to 0} f^{\e}$ is 
 an equilibrium, we deduce that $f^0$ is given by (\ref{hl}):
\begin{equation}\label{hlp}
f^0(\theta,w, \nu,t) = P_{\nu}(\theta, t) \,  M_{V(\theta,\nu,t)}(w)\,.
\end{equation}
Recalling from Lemma \ref{m2} that for any $(\theta,t)$, the function $\nu \in \mathbb{R} \mapsto 
P_{\nu}(\theta,t) \in \mathbb{R}_{+}$ belongs to $L^1(\mathbb{R})$ and meets the conditions (\ref{Pcond})--(\ref{pt});
while as \eqref{V}, 
\[\dd V=V(\theta,\nu,t) = \nu + KP(\theta,t) = \nu + K \int\limits_{\mathbb{R}} \, P_{\nu_*}(\theta,t) \, g(\nu_*) \,  \dn_*\]
in $\mathbb{R}\,,$ and the normalization (\ref{normM}) of $M_V$ is taken into account. 
In the rest of the proof, the superscript 0 is omitted for the purpose of clarity.

We first prove (\ref{hl1}).  By simple computation, we deduce from 
(\ref{normM}) that
\[\lim_{\e \to 0} \int\limits_{\mathbb{R}} f^{\e} \, \dw
= \int\limits_{\mathbb{R}} P_{\nu}(\theta,t) \, M_V(w) \, \dw = P_{\nu}(\theta,t)\,,\]
and derive from (\ref{c1}) that
\[\lim_{\e \to 0} \int\limits_{\mathbb{R}} w f^{\e} \, \dw
= \int\limits_{\mathbb{R}}  P_{\nu}(\theta,t) \, w  \, M_V(w)   \, \dw = V P_{\nu}(\theta,t) \,.\]

Let $\phi(\nu)$ be an arbitrary function in $L^{\infty}(\mathbb{R})$. 
Then, (\ref{per2}) is multiplied by $\phi(\nu)$ and integrated with respect to $(w,\nu) \in 
\mathbb{R} \times \mathbb{R}\,.$ 
With Proposition \ref{gcis}, making use of the fact that $\phi(\nu)$ is a GCI in Definition \ref{gci1},
we obtain from the right-hand side of (\ref{per2}) that
\begin{align*}
\iint\limits_{\mathbb{R} \; \mathbb{R}} Q(f^{\e}) \, \phi(\nu) \,  g(\nu) \, \dw \, \dn = 0\,.
\end{align*}
Equivalently, the right-hand side of (\ref{per2}) gives
\begin{equation}\label{ori2a}
\iint\limits_{\mathbb{R} \; \mathbb{R}}
 (\partial_t f^{\e} + \partial_{\theta}(w f^{\e})) \, \phi(\nu) \,  g(\nu) \, \dw \, \dn = 0\,. 
\end{equation}
Letting $\e \to 0$ in (\ref{ori2a}) and recalling that $f^{\e} \to P_{\nu} M_V$, we reach from \eqref{normM} and \eqref{c1} that
\begin{equation}\label{hl1a}
\int\limits_{\mathbb{R}} (\partial_t P_{\nu} + \partial_{\theta}
 (V P_{\nu})) \, \phi(\nu) \, g(\nu) \, \dn = 0\,.  
\end{equation}
Because this equation holds for all $\phi(\nu) \in L^{\infty}(\mathbb{R})$ and $g(\nu) \geq 0\,,$ the result (\ref{hl1}) 
follows: 
\begin{align}
 \partial_t P_{\nu} + \partial_{\theta}(V P_{\nu}) = 0 
 \quad \forall \nu \in \mathbb{R}\,. 
 \end{align}

 \bigskip

In order to prove (\ref{hl2}), we multiply (\ref{per2}) by $\chi_{V[f^{\e}]}(w)$ and integrate the 
resulting equation with 
respect to $w\,.$  Thanks to Proposition \ref{gcis}, $\chi_{V[f^{\e}]}(w)$ is a GCI associated to $V[f^{\e}]$ in Definition \ref{gci1}, with the constraint $V[f^0] = V$ \eqref{Vequil}; and we thus obtain
\begin{align*}
\iint\limits_{\mathbb{R} \; \mathbb{R}} Q(f^{\e}) \, \chi_{V[f^{\e}]}(w) \,   g(\nu) \, \dw \, \dn = 0\,.
\end{align*}
Equivalently,
\begin{equation}\label{ori2}
\iint\limits_{\mathbb{R} \; \mathbb{R}}
 (\partial_t f^{\e} + \partial_{\theta}(w f^{\e})) \, \chi_{V[f^{\e}]}(w) \, g(\nu) \, \dw \, \dn = 0\,. 
\end{equation}
Letting $\e \to 0$ in (\ref{ori2}) and recalling from \eqref{hlp} that $f^{\e} \to P_{\nu} M_V$, we get
\begin{equation}\label{ori3}
\iint\limits_{\mathbb{R} \; \mathbb{R}}
 (\partial_t(P_{\nu} M_V) + \partial_{\theta}(w P_{\nu} M_V)) \, \chi_V(w)  \, g(\nu) \, \dw \, \dn = 0\,.
\end{equation}

Now, by (\ref{sol}), we have 
\[\chi_V (w) = w-V\,.\]
Therefore, (\ref{ori3}) is equivalent to
\begin{align*}
\iint\limits_{\mathbb{R} \; \mathbb{R}} (M_V \, \partial_t (P_{\nu}) 
+ P_{\nu} \, \partial_t (M_V) +  w \, M_V \, \partial_{\theta} (P_{\nu}) 
 + w \, P_{\nu} \, \partial_{\theta} (M_V)) \, (w-V)  \, g(\nu) \, \dw \, \dn = 0\,.
\end{align*}
That is,
\begin{align}\label{eqa}
\begin{split}
&\iint\limits_{\mathbb{R} \; \mathbb{R}}
\left(M_V \, \partial_t (P_{\nu}) 
 + (w-V) \, P_{\nu} \, M_V \, \partial_t V 
 + w \, M_V \,\partial_{\theta} (P_{\nu})\right.\\
 & \left. \hspace{155pt} + w\, (w-V)\, P_{\nu} \, M_V \, \partial_{\theta}V\right) 
  (w-V)  \, g(\nu) \, \dw \, \dn = 0\,.
  \end{split}
\end{align} 
Integrating (\ref{eqa}) by parts leads to equation (\ref{hl2}) as desired, that is,
\begin{align}
 P \, \partial_t V +  P \, (Y+KP) \,  \partial_{\theta} V + \partial_{\theta} P = 0\,.
\end{align}

\noindent Indeed, first, by (\ref{m0}),
\begin{align*}
\iint\limits_{\mathbb{R} \; \mathbb{R}}
 (w-V) \, M_V \, \partial_t (P_{\nu}) \,  g(\nu) \, \dw \, \dn = 0\,.
\end{align*}

\noindent Second,
\begin{align*}
\iint\limits_{\mathbb{R} \; \mathbb{R}}
(w-V)^2 M_V \, P_{\nu} \,  \partial_t (V) \,  g(\nu) \, \dw \, \dn
= \int\limits_{\mathbb{R}} P_{\nu} \, \partial_t (V)  \, g(\nu) \, \dn = P \, \partial_t V\,.
\end{align*}

\noindent Third,
\begin{align*}
\iint\limits_{\mathbb{R} \; \mathbb{R}}
w  (w-V) \, M_V \, \partial_{\theta} (P_{\nu}) \,  g(\nu) \, \dw \, \dn
=\int\limits_{\mathbb{R}} \partial_{\theta} (P_{\nu}) \,  g(\nu) \, \dn 
=\partial_{\theta} P\,.
\end{align*}

\noindent Lastly, using \eqref{normM} and \eqref{c1}, we get 
\begin{align*}
\begin{split}
\iint\limits_{\mathbb{R} \; \mathbb{R}} & w(w-V)^2 \, M_V \, P_{\nu}  \, \partial_{\theta}(V) \,  g(\nu) \, \dw \, \dn
=\int\limits_{\mathbb{R}} P_{\nu} \, V \, \partial_{\theta} (V) \,  g(\nu) \, \dn \\
&= \int\limits_{\mathbb{R}} (P_{\nu}) \,  (\nu+KP) \, (\partial_{\theta}V)  \, g(\nu) \, \dn 
= P \, (Y + KP) \, \partial_{\theta}V\,.
\end{split}
\end{align*}
\end{proof}


\section{Conclusion}\label{conclude}
In this paper, we have investigated the Kuramoto--Sakaguchi--Fokker--Planck equation (namely, parabolic Kuramoto--Sakaguchi, or Kuramoto--Sakaguchi equation),
having inertia and white noise effects.  We derived the hydrodynamic limit for this Kuramoto--Sakaguchi equation.  While showing this major result, as a complement, we also proved a Hardy-type inequality over the entire real line.  

\appendix

\section{Preliminary results about absolute continuity}\label{hardytype}
For the purpose of stating our Hardy-type inequality \eqref{hardyineq}, we need the following preliminaries regarding continuity.

\begin{definition}[\cite{leonipde,folland,brezis},  \textbf{absolute continuity}]\label{ACdef}
Let $J \subseteq \mathbb{R}$ be an interval. A function $F: J \to \mathbb{R}$ is called \textbf{absolutely continuous} 
if for every $\e >0\,,$ there exists $\delta >0$ such that for any finite sequence of pairwise disjoint subintervals (which are non-overlapping or have disjoint interiors) $(a_1,b_1), \ldots, (a_l,b_l) \subseteq J$ with $a_k < b_k \in J\,,$ 
%
%
\begin{equation}\label{acdef1}
\sum_{k=1}^l (b_k - a_k) < \delta \implies \sum_{k=1}^l |F(b_k) - F(a_k)| < \e \,.
\end{equation}
Note that as $l$ is arbitrary, one may also let $l= \infty\,,$ that is, replacing finite sum with series from countably infinite collection.
Here, $J$ can be $[a,b]$ or $(a,b)\,,$ and especially, this definition \ref{ACdef} includes the case $J=\mathbb{R}$ (as in \cite[p.~135]{folland} or \cite[p.~105]{cohnm}).  
%

%

We denote the space of all absolutely continuous functions $F: J \to \mathbb{R}$ by $AC(J)\,:$ \begin{equation}\label{ACrep}  
AC(J) := \{F: J \to \mathbb{R} \ \big | \ F \tu{ is absolutely continuous on } J\}\,,
\end{equation}
and
\[AC([a,b]) = AC[a,b]\,, \qquad AC((a,b)) = AC(a,b)\,.\]

\bigskip

For $J=[a,b]\,,$ by the Fundamental Theorem of Calculus for Lebesgue Integrals \cite[Theorem 3.35]{folland}, if $-\infty < a < b < \infty$ and $F: [a,b] \to \mathbb{R}\,,$ then for all $x \in [a,b]\,,$ the following are equivalent:
\begin{enumerate}[(a)]
\item $F$ is absolutely continuous on $[a,b]\,.$ \label{acftc}
\item $F(x) - F(a) =  \dd \int_a^x f(t) \, \dt$ for some $f \in L^1([a,b],m)\,.$ \label{aci}
\item $F$ is differentiable (and has a derivative $F'$) a.e.\ on $[a,b]\,,$ $F' \in L^1([a,b],m)\,,$ and $F(x) - F(a) = \dd \int_a^x F'(t) \, \dt\,.$ \label{acd}
\end{enumerate}
If these equivalent conditions hold, then necessarily any function $f$ as in condition \ref{aci} satisfies $f=F'$ a.e.\ in $[a,b]\,.$ 
\bigskip

\begin{definition}[\cite{leonipde},  \textbf{locally absolute continuity}]\label{locACdef}
Let $J \subseteq \mathbb{R}$ be an interval. A function $F: J \to \mathbb{R}$ is called \textbf{locally absolutely continuous} if it is absolutely continuous in the sense of \eqref{acdef1} on $[\lambda, \gamma]$ for every compact subinterval $[\lambda,\gamma] \subseteq J\,.$  
\end{definition}
The collection of all locally absolutely continuous functions $F: J \to \mathbb{R}$ is denoted by $AC_{\tu{loc}}(J)\,.$ That is,
\[AC_{\tu{loc}}(J):=\{F: J \to \mathbb{R} \, \big | \, F \in AC[\lambda,\gamma] \tu{ for every compact subinterval } [\lambda,\gamma] \subseteq J\}\,.\]

Also, benefiting from the notation in \cite[Definition 1.2 on p.~5--6]{hardyopic}, we denote
\[AC_{\tu{loc}}([a,b]) = AC_{\tu{loc}}[a,b]=AC[a,b] \,, \qquad AC_{\tu{loc}}((a,b)) =AC_{\tu{loc}}(a,b)\,.\]
Moreover, let
\begin{align}\label{aclr}
\begin{split}
AC_{\tu{L}}(J) &= AC_{\tu{L}}(a,b):=\{ F \in AC_{\tu{loc}}(J) \ | \ \lim_{w \to a^{+}} F(w) = 0\}\,,\\ 
AC_{\tu{R}}(J) &=AC_{\tu{R}}(a,b):=\{ F \in AC_{\tu{loc}}(J) \ | \ \lim_{w \to b^{-}} F(w) = 0\}\,.
\end{split}
\end{align}
Specially, we denote
\begin{equation}\label{aclocr}
AC_{\tu{loc}}(\mathbb{R})=\{F: \mathbb{R} \to \mathbb{R} \, \big | \, F \in AC[\lambda,\gamma] \tu{ for every compact interval } [\lambda,\gamma] \subset \mathbb{R}\}\,,
\end{equation}
which is called the space of locally absolutely continuous functions over $\mathbb{R}\,.$ 
In this paper, we use the convention that if a function belongs to $AC(\mathbb{R})$ (defined in \eqref{ACrep}) then it belongs to $AC_{\tu{loc}}(\mathbb{R})$ (defined by \eqref{aclocr}), that is $AC(\mathbb{R}) \subset AC_{\tu{loc}}(\mathbb{R})\,.$

%

\end{definition}
%
%
%

\bigskip
Recall that our considered space is \eqref{w0}:
\begin{equation}\label{w00}
W_0 = \left\{\varphi \in H^1(\mathbb{R}) \cap L^1(\mathbb{R}) \,, \int\limits_{\mathbb{R}} \varphi(w) \, \dw = 0\right\}\,.
\end{equation}
To work with this space, where $H^1(\mathbb{R})=W^{1,2}(\mathbb{R})\,,$ it is worth noting that for any open set $\Omega \subseteq \mathbb{R}$ and any $1 \leq p \leq \infty\,,$ we have $L^p(\Omega) \subset L^p_{\tu{loc}}(\Omega)$ \cite[p.~67]{boyer_boch_der} 
as well as $L^p_{\tu{loc}}(\Omega) \subset L^1_{\tu{loc}}(\Omega)$ (for $1 < p \leq \infty$) \cite[p.~106]{brezis}, and thus $L^p(\Omega) \subset L^1_{\tu{loc}}(\Omega)\,.$  Consequently, $H^1(\mathbb{R}) \subset H^1_{\tu{loc}}(\mathbb{R})$ (see the representations of local Sobolev spaces in \cite[Definition 7.14 on p.~188]{leonipde} or \cite[Definition III.2.2 on p.~136]{boyer_boch_der}).  

\bigskip

On the real line, the connection between Sobolev functions and absolutely continuous functions
is described in \cite[Theorem 7.16 on p.~189]{leonipde} (with proof), and we restate it more specifically in our context, for the sake of completeness.

\begin{theorem}[Theorem 7.16 in \cite{leonipde}]\label{mainH1AC}
Let $J \subseteq \mathbb{R}$ be an open set and let $1 \leq p \leq \infty\,.$  Then, a function $u: J \to \mathbb{R}$ belongs to $W^{1,p}(J)$ if and only if it admits an absolutely continuous representative $\overline{u}: J \to \mathbb{R}$ (with $u=\overline{u}$ almost everywhere) such that both $\overline{u}$ and its classical derivative $\overline{u}'$ belong to $L^p(J)\,.$  Furthermore, if $p>1\,,$ then $\overline{u}$ is H\"{o}lder continuous of exponent $1/p'\,,$ where $1/p + 1/p'=1\,.$   
\end{theorem}

By considering each connected component of $\Omega\,,$ one can assume $\Omega = J\,,$ where $J \subseteq \mathbb{R}$ is an open interval. In this manner, a version of \cite[Theorem 7.16]{leonipde}
is presented in \cite[Appendix C, Chapter 17, Theorem 17.15 on p.~647]{leonifrac} over an open interval $J \subseteq \mathbb{R}\,,$ also by Leoni.
%
Therefore, letting $J = \mathbb{R}$ in Theorem \ref{mainH1AC}, and using the fact that $AC(\mathbb{R}) \subset AC_{\tu{loc}}(\mathbb{R})$ (where $AC(\mathbb{R})$ and $AC_{\tu{loc}}(\mathbb{R})$ are respectively defined in \eqref{ACrep} and \eqref{aclocr}), we directly derive the following result from the proof of \cite[Theorem 7.16]{leonipde}. 
\begin{cor}\label{corH1AC}
A function $\varphi: \mathbb{R} \to \mathbb{R}$ belongs to $H^1(\mathbb{R})$ if and only if it admits a locally absolutely continuous representative $\overline{\varphi}: \mathbb{R} \to \mathbb{R}$ (with $u=\overline{u}$ almost everywhere) such that both $\overline{\varphi}$ and its classical derivative $\overline{\varphi}'$ belong to $L^2(\mathbb{R})\,.$
Moreover, $\overline{\varphi}$ is H\"{o}lder continuous of exponent $1/2$ over all $\mathbb{R}\,,$ and thus $\overline{\varphi}$ is uniformly continuous and continuous on $\mathbb{R}\,.$ 
\end{cor}

\begin{proof}
The proof is laborious with many steps and can be found in \cite[Theorem 7.16]{leonipde}, for the connection between Sobolev functions and absolutely continuous functions over all $\mathbb{R}\,.$  This result is also presented in \cite[p.~259, or Exercise 5.4 on p.~306]{evans} with a statement for an open interval $(a,b) \subset \mathbb{R}$ (strict inclusion), as well as \cite[Theorem 4.20 on p.~188]{evans15} with a proof for local Sobolev spaces $W_{\tu{loc}}^{1,p}(\mathbb{R})\,,$ where $1\leq p < \infty\,.$  

Regarding the continuity of $\overline{\varphi}\,,$ by the Fundamental Theorem of Calculus 
and H\"{o}lder's inequality, for $w > y\,,$ we have
\begin{equation}\label{holdercont}
|\overline{\varphi}(w) - \overline{\varphi}(y)| = \left | \int_y^w \overline{\varphi}'(s) \, \ds \right | \leq \int_y^w |\overline{\varphi}'(s)| \, \ds \leq \|\overline{\varphi}'\|_{L^2(\mathbb{R})} \, |w-y|^{1/2}\,.
\end{equation}
Hence, $\overline{\varphi}$ is H\"{o}lder continuous with exponent $1/2$ over all $\mathbb{R}\,,$ which implies that $\overline{\varphi}$ is also uniformly continuous on $\mathbb{R}\,.$ 
Indeed, for every real number $\e >0\,,$ there exists a real number $\delta >0$ satisfying $\|\overline{\varphi}'\|_{L^2(\mathbb{R})} \,  \delta^{1/2} < \e$ such that for every $w,y \in H^1(\mathbb{R})$ with $|w-y| < \delta\,,$ we obtain from \eqref{holdercont} that 
\[|\overline{\varphi}(w) - \overline{\varphi}(y)| \leq  \|\overline{\varphi}'\|_{L^2(\mathbb{R})} \, |w-y|^{1/2} < \|\overline{\varphi}'\|_{L^2(\mathbb{R})} \, \delta^{1/2} < \e\,.\]    
Hence, $\overline{\varphi}$ is also uniformly continuous and thus continuous over all $\mathbb{R}\,.$

A simpler way to show the continuity of $\overline{\varphi}$ on $\mathbb{R}$ is as follows.  With $\overline{\varphi} \in AC_{\tu{loc}}(\mathbb{R})\,,$ by Definition \eqref{aclocr}, $\overline{\varphi}$ is absolutely continuous on every compact subinterval of $\mathbb{R}\,.$ Thus, $\overline{\varphi}$ is continuous on every compact subinterval of $\mathbb{R}\,,$ so $\overline{\varphi}$ is continuous in $\mathbb{R}$ (because $\mathbb{R}$ is a complete metric space with the metric $d(x,y) = |x-y|$).
\end{proof}
%

\begin{remark}\label{idh1ac}
Corollary \ref{corH1AC} means that there exists a representative $\overline{\varphi} \in H^1(\mathbb{R})$ (still conveniently denoted by $\varphi$) which is locally absolutely continuous on $\mathbb{R}\,.$  

Also note that the characteristic ``$\varphi$ has a continuous representative $\overline{\varphi}$'' (with $\varphi = \overline{\varphi}$ a.e.\ in $\mathbb{R}$) is not the same as ``$\varphi$ is continuous a.e.'' \cite[Remark 5 on p.~204]{brezis}.
\end{remark}

\bigskip
Some further properties of $\varphi \in W_0$ defined in \eqref{w0} are as follows.
\begin{proposition}\label{d0}
Given $\varphi \in W_0\,,$ there is some $d \in \mathbb{R}$ such that $\varphi(d) = 0\,.$  Moreover, there exist $\dd \lim_{w \to d^{+}}\varphi(w) = \lim_{w \to d^{-}}\varphi(w) = \lim_{w \to d}\varphi(w) = \varphi(d) = 0\,.$
\end{proposition}

\begin{proof}
With $\varphi \in W_0$ (having $\varphi = \overline{\varphi}$ by Remark \ref{idh1ac}), where $\dd \int\limits_{\mathbb{R}} \varphi(w) \, \dw = 0$ (the average value of $\varphi(w)$ over $\mathbb{R}$ is zero), it holds that there exists an element $d \in \mathbb{R}$ such that $\varphi(d) = 0\,.$  Indeed, assume that $\varphi(w) \neq 0$ for any $w \in \mathbb{R}\,.$  Since $\varphi(w)$ is continuous on $\mathbb{R}$ (from Corollary \ref{corH1AC}), it follows that $\varphi$ has no jumps, and it can be always positive or always negative.  Without loss of generality, assume $\varphi$ is always positive (and the case $\varphi$ is always negative is analogous). But then, with positive $\varphi\,,$ the integral $\dd \int_{-\infty}^{\infty} \varphi(w) \, \dw$ is positive.  This contradicts the given condition  $\dd \int_{-\infty}^{\infty} \varphi(w) \, \dw=0\,.$  Thus, there is some $d \in \mathbb{R}$ such that $\varphi(d) =0\,.$ 

Furthermore, by the continuity of $\varphi$ over all $\mathbb{R}\,,$ there exist $\dd \lim_{w \to d^{+}}\varphi(w) = \lim_{w \to d^{-}}\varphi(w) = \lim_{w \to d}\varphi(w) = \varphi(d) = 0\,.$
\end{proof}

\begin{proposition}\label{bounded0infty}
For any $1<p<\infty\,,$ if $\varphi \in W^{1,p}(\mathbb{R})\,,$ then $\varphi$ is bounded, and $\dd \lim_{w \to \pm \infty} \varphi(w) = 0\,.$ 
\end{proposition}
\begin{proof}
A proof for $\varphi$ being bounded, that is, $\|\varphi\|_{L^{\infty}(\mathbb{R})} \leq C \| \varphi \|_{W^{1,p}(\mathbb{R})}$ (with $1\leq p \leq \infty$) can be found in \cite[Theorem 8.8 on p.~212]{brezis}), using H\"{o}lder's inequality and Young's inequality.  Also, a proof for $\dd \lim_{w \to \pm \infty} \varphi(w) = 0$ (with $1 \leq p < \infty$) is presented in \cite[Corollary 8.9 on p.~214]{brezis}), using density. 
\end{proof}
%
%

%

Consider $M_V(w) = \dfrac{1}{\sqrt{2\pi}} \tu{exp} \left( -\dfrac{1}{2}(w-V)^2 \right)$ defined in \eqref{equi} and $\varphi \in W_0$ specified in \eqref{w0}.  We have the following results.  
\begin{proposition}\label{muac}
Both $M_V(w)$ and $\dd u(w) = \frac{\varphi(w)}{M_V(w)}$ are in $AC_{\tu{loc}}(\mathbb{R})\,.$
\end{proposition}
\begin{proof}
For every bounded closed interval $[a,b] \subset (-\infty,\infty)\,,$ we know that $k(w) = -  (w-V)^2 / 2$ is continuous on $[a,b]\,,$ so it is bounded, say $|k(w)| \leq \alpha \,, \forall w \in [a,b]\,,$ for some constant $\alpha \geq 0\,.$ Since the function $y \to e^y$ is smooth (so $C^1$ function) on the compact interval $[-\alpha,\alpha]$ (where $y \in [-\alpha,\alpha]$), it follows that $e^y$ is Lipschitz continuous
%
%
on $[-\alpha, \alpha]\,,$ and thus $e^{k(w)}$ is absolutely continuous on $[a,b]\,.$  Hence, $M_V(w)$ is in $AC_{\tu{loc}}(\mathbb{R})\,.$ 

Now, note that given $\varphi$ and $g$ which are absolutely continuous functions on every bounded closed interval $[a,b]\,,$ and provided that the function $g$ is defined and nowhere zero on $[a,b]\,,$ it holds that the reciprocal of $g$ is absolutely continuous on $[a,b]$ (see \cite[Problem 5.14(c) on p.~111]{royden88}), and the product $\varphi \cdot (1/g)$ is absolutely continuous on $[a,b]$ (see \cite[Problem 5.14(b) on p.~111]{royden88}).  Using this remark, for any $\varphi \in W_0$ specified in \eqref{w0}, with $g(w) = M_V(w)$ defined by \eqref{equi}, the function  
\begin{equation}\label{upm}
u(w) = \frac{\varphi(w)}{M_V(w)}
\end{equation}
is absolutely continuous on any compact subinterval $[a,b]$ in $\mathbb{R}\,,$ and hence $u(w)$ is \textbf{locally absolutely continuous} over $\mathbb{R}\,.$  That is, $u(w) \in AC_{\tu{loc}}(\mathbb{R})$ defined in \eqref{aclocr} (see also \cite{frank22}). 
\end{proof}
%

\section{Improper Riemann integrals and Lebesgue integrals}\label{appii}
We now show that for $\varphi \in W_0$ \eqref{w0}, the Lebesgue integral equals the improper Riemann integral:
\begin{equation}\label{ler}
\int_{\mathbb{R}} \varphi \, \dm = \int_{-\infty}^{\infty} \varphi(w) \, \dw < \infty\,.
\end{equation}

Over a compact interval, such relation between Riemann integral and Lebesgue integral  can be found in \cite[Theorem 2.28 on p.~57]{folland}. 
Over all $\mathbb{R}\,,$ the relation \eqref{ler} can be derived from \cite{apostol} in a standard manner.
However, for the completeness in our context of space $W_0$ $\eqref{w0}\,,$ the proof for \eqref{ler} is presented in detail here.  To this end, we need the following preliminaries.

\begin{definition}[Doubly infinite integral \cite{apostol}]\label{dii}
It is said that 
\begin{equation}\label{diiform}
\int_{-\infty}^{\infty} \varphi(w) \, \dw
\end{equation}
is \textbf{convergent} if both
\begin{equation}\label{2ii}
\int_{-\infty}^{0} \varphi(w) \, \dw \text{ and } \int_{0}^{\infty} \varphi(w) \, \dw
\end{equation}
are convergent, that is, if the limits 
\begin{equation}\label{2iilim}
\lim_{s \to \infty} \int_{-s}^{0} \varphi(w) \, \dw \text{ and } \lim_{s \to \infty} \int_{0}^{s} \varphi(w) \, \dw 
\end{equation}
both exist (for all $s \geq 0$).  In such case, 
\begin{equation}\label{isum2}
\int_{-\infty}^{\infty} \varphi(w) \, \dw = \lim_{s \to \infty} \int_{-s}^{0} \varphi(w) \, \dw + \lim_{s \to \infty} \int_{0}^{s} \varphi(w) \, \dw \,.
\end{equation}
If any one of these limits does not exist, then it is said that $\dd \int_{-\infty}^{\infty} \varphi(w) \, \dw$ is divergent.
\end{definition}

We thus note that an improper integral that does not converge is called \textbf{diverges}, and may simply \textbf{diverge} in no specific direction (as divergence by oscillation, or as $-1,1,-1,1,\ldots$). 
It can also happen that an improper integral \textbf{diverges (or tends) to infinity}; and in such case, it may be assigned the value of $\infty\,.$
\begin{lemma}\label{symlim}
If $\varphi: \mathbb{R} \to \mathbb{R}$ is Riemann integrable on every compact subinterval of $\mathbb{R}$ and the integral 
\[\int_{-\infty}^{\infty} \varphi(w) \, \dw \]
exits, then the \textbf{symmetric limit} \cite{apostol} 
\[\lim_{s \to \infty} \int_{-s}^s \varphi(w) \, \dw\]
exists and 
\[\lim_{s \to \infty} \int_{-s}^s  \varphi (w)  \, \dw = \int_{-\infty}^{\infty}  \varphi (w)  \, \dw\,.\]
\end{lemma}

\begin{proof}
Suppose $\dd \int_{-\infty}^{\infty} \varphi(w) \, \dw$ exists.  Then by Definition \ref{dii}, 
\[\lim_{s \to \infty} \int_{-s}^{0} \varphi(w) \, \dw \text{ and } \lim_{s \to \infty} \int_{0}^{s} \varphi(w) \, \dw \] both exist (for all $s \geq 0$).  Hence, 
\begin{align*}
\lim_{s \to \infty } \int_{-s}^{s} \varphi(w) \, \dw &= \lim_{s \to \infty} \left( \int_{-s}^{0} \varphi(w) \, \dw +  \int_{0}^{s} \varphi(w) \, \dw \right)\\
&= \lim_{s \to \infty}  \int_{-s}^{0} \varphi(w) \, \dw + \lim_{s \to \infty} \int_{0}^{s} \varphi(w) \, \dw = \int_{-\infty}^{\infty} \varphi(w) \, \dw\,.
\end{align*}
\end{proof}

\begin{remark}\label{cpvrmk}
It is important to note that when the symmetric limit $\dd \lim_{s \to \infty} \int_{-s}^s \varphi(w) \, \dw$ exists but $\dd \int_{-\infty}^{\infty} \varphi(w) \, \dw$ is divergent (for instance, choose $\varphi(w) = w$), then the symmetric limit is called the \textbf{Cauchy principal value} of $\dd \int_{-\infty}^{\infty} \varphi(w) \, \dw\,.$  For example, $\dd \int_{-\infty}^{\infty} w \, \dw$ has the Cauchy principal value 0, but the integral $\dd \int_{-\infty}^{\infty} w \, \dw$ 
 does not exist.  
\end{remark}

\begin{lemma}\label{irisl}
If a non-negative function $g:\mathbb{R} \to [0,\infty)$ is Riemann integrable on every compact subinterval of $\mathbb{R}$ and the symmetric limit
\[\lim_{s \to \infty} \int_{-s}^s  g (w)  \, \dw\]
exists, then
$\dd \int_{-\infty}^{\infty} g(w) \, \dw$ exists and
\begin{equation}\label{irisle}
\int_{-\infty}^{\infty}  g(w)  \, \dw = \lim_{s \to \infty} \int_{-s}^s  g(w)  \, \dw\,. 
\end{equation}
\end{lemma}

\begin{proof}
For all $s \geq 0\,,$ since $g$ is non-negative, it follows that $\dd h(s) = \int_0^s g(w) \, \dw$ is an increasing function on $[0,\infty)\,.$  More specifically,
\[h(s) = \int_0^s g(w) \, \dw \leq \int_{-s}^s g(w) \, \dw \leq \lim_{s \to \infty} \int_{-s}^s g(w) \, \dw < \infty\,.\]
That is, $h(s)$ is bounded above, so 
\[\lim_{s \to \infty} \int_0^s g(w) \, \dw = \lim_{s \to \infty} h(s) = \sup \{ h(s): s\in [0, \infty)\}\,, \]
by the completeness of the real line.
Similarly, for $s \in [0,\infty)\,,$ it can be proved that $\dd k(s) = \int_{-s}^0 g(w) \, \dw$ is an increasing function, which is bounded above by $\dd \lim_{s \to \infty} \int_{-s}^s g(w) \, \dw < \infty\,.$  Hence,
\[\lim_{s \to \infty} \int_{-s}^0 g(w) \, \dw = \lim_{s \to \infty} k(s) = \sup \{ k(s): s\in [0, \infty)\}\,. \]
Therefore, $\dd \int_{-\infty}^{\infty} g(w) \, \dw$ exists (by Definition \ref{dii}).  Using Lemma \ref{symlim}, we have 
\[\int_{-\infty}^{\infty}  g(w)  \, \dw = \lim_{s \to \infty} \int_{-s}^s  g(w)  \, \dw\,.\]
The proof is completed.
\end{proof}
\begin{remark}\label{isosl}
If $\varphi:\mathbb{R} \to \mathbb{R}$ is Riemann integrable on every compact subinterval of $\mathbb{R}$ and
\begin{equation}\label{airi}
\int_{-\infty}^{\infty} | \varphi (w) | \, \dw < \infty\,,
\end{equation}
then the integral
\begin{equation}\label{airi1}
\int_{-\infty}^{\infty}  \varphi (w)  \, \dw
\end{equation}
is called \textbf{absolutely convergent}, and it is also convergent (or exists); and thus by Lemma \ref{symlim}, the symmetric limit
\[\lim_{s \to \infty} \int_{-s}^s \varphi(w) \, \dw\]
exists, and 
\begin{equation}\label{irisle1}
\int_{-\infty}^{\infty}  \varphi (w)  \, \dw = \lim_{s \to \infty} \int_{-s}^s  \varphi (w)  \, \dw\,. 
\end{equation}
\end{remark}

\begin{lemma}\label{legs}
If $\varphi \in C([a,b])\,,$ then
\begin{equation}\label{addlegs}
\tilde{\varphi}(w) = \begin{cases}
 \varphi(w) &\text{ if } w\in[a,b]\\
 0 &\text{ otherwise }
\end{cases}
\end{equation}
is a Lebesgue integrable function over all $\mathbb{R}\,.$
\end{lemma}
%
\begin{proof}
This result is obtained from \cite[Theorem 10.18 on p.~263]{apostol}, and the proof is left to the reader.    
\end{proof}
\begin{lemma}\label{irl}
A continuous function $\varphi \in C(\mathbb{R})$ is Lebesgue integrable if and only if the ``improper Riemann integral''
\begin{equation}\label{iraf}
\lim_{s \to \infty} \int_{-s}^s | \varphi(w)| \, \dw < \infty  \,.  
\end{equation}
Moreover, if this result holds, then both $\varphi$ and $| \varphi |$ are improper Riemann-integrable over all $\mathbb{R} = (-\infty,\infty)\,;$ 
and the Lebesgue integral of $\varphi$ is equal to the improper Riemann integral of $\varphi$ on $\mathbb{R}$ as \eqref{ler}:
\begin{equation}\label{ler2}
\int_{\mathbb{R}} \varphi \, \dm = \int_{-\infty}^{\infty} \varphi(w) \, \dw < \infty\,.
\end{equation}
\end{lemma}

\begin{proof}
The proof on the domain $[a,\infty)$ can be found in \cite[Theorems 10.31 and 10.33]{apostol}.  Now, we prove this Lemma over all $\mathbb{R} = (-\infty, \infty)\,.$ Provided that $\varphi$ is continuous (so $|\varphi|$ is continuous as well) in all $\mathbb{R}\,.$ Thus, $\varphi$ and $|\varphi|$ are both Riemann integrable on every compact (closed and bounded) subinterval of $\mathbb{R}\,.$  
%
%

For the first direction of the Lemma, if $\varphi \in L^1(\mathbb{R})\,,$ then $|\varphi| \in L^1(\mathbb{R})\,,$ and $\chi_{[-s,s]} |\varphi| \in L^1(\mathbb{R})$ converges to $|\varphi| \in L^1(\mathbb{R})\,.$  Thus, by Dominated Convergence Theorem, \eqref{iraf} must be satisfied.

Conversely, if \eqref{iraf} holds, then consider the sequence of continuous functions $\chi_{[-s,s]}|\varphi|\,,$ which is known to be in $L^1(\mathbb{R})$ by Lemma \ref{legs}.  This sequence of non-negative functions is monotonically increasing to $|\varphi|\,.$ Thus, by the Monotone Convergence theorem, it follows that $|\varphi| \in L^1(\mathbb{R})\,.$ 
Next, consider the sequence of continuous functions $\chi_{[-s,s]}\varphi\,,$ which we know to be in $L^1(\mathbb{R})$ by Lemma \ref{legs}. 
 Since this sequence $\chi_{[-s,s]}\varphi$ is bounded by $|\varphi|$ and converges pointwise to $\varphi\,,$ it follows from the Dominated Convergence Theorem that $\varphi \in L^1(\mathbb{R})$ and 
 \[\int_{\mathbb{R}} \varphi \, \dm = \lim_{s \to \infty} \int_{\mathbb{R}} \chi_{[-s,s]}\varphi \, \dm = \lim_{s \to \infty} \int_{-s}^s \varphi(w) \, \dw = \int_{-\infty}^{\infty}  \varphi (w)  \, \dw\,,\]
where the last equality is satisfied thanks to the hypothesis \eqref{iraf} and Remark \ref{isosl}.  Therefore, \eqref{ler2} (as \eqref{ler}) holds, and this completes the proof.
\end{proof}

%

\section{Hardy's inequality with weights over a half-open interval}\label{hiw1}

Our Hardy-type inequality \eqref{hardyineq} over $(-\infty, \infty)$ is developed from the Hardy's inequality in \cite[Theorem~1.14]{hardyopic} or \cite[Theorem 4]{Muck1972} by Muckenhoupt over a half-open interval $[d,\infty)\,,$ which is restated in Theorem \ref{ttam} below.

We first recall weight function terminology (see \cite[Definition 1.4 on p.~8]{hardyopic}, for instance).

\begin{definition}\label{weightdef}
A \textit{weight function} $\mu$ on $(a,b)$ with $-\infty \leq a < b \leq \infty$ is defined to be a function that is measurable, positive, and finite almost everywhere (a.e.) on $(a,b)\,.$ 
\end{definition}

Also, recall the definition $AC_{\tu{L}}((a,b))$ from \eqref{aclr}.

\bigskip

The following statement is similar to a well-known result by Tomaselli, Talenti, and Artola. 
\begin{theorem}\label{ttam}
If $\mu$ and $\nu$ are continuous weight functions on $(d,\infty)$ with rapid decay at infinity, and $u \in AC_{\tu{L}}([d, \infty))$ with $u(d)=0$ for some $d \in \mathbb{R}\,,$ then the Hardy's inequality
\begin{equation}\label{hardyineq1mn}
\int_{d}^{\infty}(u(w))^2 \, \mu(w) \, \dw  \leq \lambda_{\tu{L}} \int_{d}^{\infty} (u'(w))^2 \, \nu(w) \, \dw 
\end{equation}
holds for some finite constant $\lambda_{\tu{L}} >0$ if and only if
\begin{equation}\label{BL1mn}
b_{\tu{L}} = \sup_{0 < r < \infty} \left( \int_r^{\infty} \mu(w) \, \dw \right) \left( \int_d^r \frac{1}{\nu(w)} \, \dw \right) < \infty\,.
\end{equation}
(Within \eqref{BL1mn}, it is known that $1/\nu(w) = +\infty$ if $\nu(w) = 0\,;$ and for the product inside the supremum, $0\cdot \infty$ implies $0\,.$)  Moreover, in the Hardy's inequality \eqref{hardyineq1mn}, the best (least) constant $\lambda_{\tu{L}}$ satisfies 
\begin{equation}\label{bestBL}
b_{\tu{L}} \leq \lambda_{\tu{L}} \leq 4b_{\tu{L}}\,.
\end{equation}
\end{theorem}  
%
We note that the supremum in \eqref{BL1mn} is taken over $d < r < \infty$ (see \cite[(1.18)]{hardyopic} or \cite{frank22}, for instance).  This Theorem \ref{ttam} provides an explicit method for determining whether the Hardy's inequality \eqref{hardyineq1mn} holds, and yields the best (smallest)
constant up to a factor of 4.  Theorem \ref{ttam} is commonly referred to as Muckenhoupt's theorem, and Muckenhoupt provided a very clear proof of it in \cite[Theorem 4]{Muck1972}, then further explanations can be found in \cite{canizo_muck}. In that article \cite{Muck1972}, one can also find references to less well-known papers by Tomaselli, Talenti, and Artola.  There are many generalizations of the Hardy's inequality \eqref{hardyineq1mn}, which can be found in the nice book by Opic and Kufner \cite{hardyopic}, where different versions of the ``basic'' Hardy's inequality \cite{Hardy1920} can be grouped together, and each inequality is referred to as a \textit{Hardy-type inequality} \cite{hardyopic}.  


Following the proof provided by Muckenhoupt \cite{Muck1972}, which is explained in \cite{canizo_muck}, we now prove the ``if'' part of Theorem \ref{ttam} (as we will need it to show our Theorem \ref{hardylem1} later).

\begin{proof}[\textbf{Proof of the ``if'' part of Theorem \ref{ttam}}.]
Choose a general function $\kappa(r)\,,$ to be specified then, where $d < r <\infty\,.$  For all $w \in [d, \infty)\,,$ by the hypothesis $u(d) = 0$ and H\"{o}lder's inequality, we obtain 
\begin{align*}
(u(w))^2 &= \left( \int_d^w u'(r) \, \dr \right)^2 = \left( \int_d^w \frac{1}{\sqrt{\nu(r)\,  \kappa(r)}} \cdot \sqrt{(u'(r))^2 \, \nu(r) \, \kappa(r)}  \, \dr \right)^2\\ 
& \leq \left(\int_d^w \left( \frac{1}{\sqrt{\nu(r)\,  \kappa(r)}} \right)^2 \, \dr \right) \left( \int_d^w \left( \sqrt{ (u'(r))^2 \, \nu(r) \, \kappa(r) } \right)^2 \, \dr \right) \\
&= \chi(w) \int_d^w (u'(r))^2 \, \nu(r) \, \kappa(r) \, \dr\,,
\end{align*}
where
\begin{equation*}
\chi(w): = \int_d^w  \frac{1}{\nu(r)\,  \kappa(r)} \, \dr \quad \forall w\in [d,\infty)\,. \end{equation*}
Therefore, using \cite[(3.2) on p.~22]{hardyopic} (which is Fubini's theorem with a change of the limits of integration), we get
\begin{align*}
\int_d^{\infty} (u(w))^2 \, \mu(w) \, \dw & \leq \int_d^{\infty} \mu(w) \, \chi(w) 
 \left( \int_d^w  (u'(r))^2 \, \nu(r) \, \kappa(r) \, \dr \right) \, \dw  \\
 &= \int_d^{\infty}  (u'(r))^2 \, \nu(r) \, \kappa(r) \left(\int_r^{\infty} \mu(w) \, \chi(w) \, \dw \right) \, \dr\,.
\end{align*}

For the purpose of proving inequality \eqref{hardyineq1mn}, it is sufficient to find a function $\kappa(r)\,,$ with $d<r<\infty\,,$ such that
\begin{align*}
\kappa(r) \, \int_r^{\infty} \mu(w) \, \chi(w) \, \dw \leq 4b_{\tu{L}} \quad \forall d<r<\infty\,.  
\end{align*}
A possible choice of $\kappa(r)$ is 
\begin{equation*}
(\kappa(r))^2 := \int_d^r \frac{1}{\nu(y)} \, \dy \quad \forall r\in [d,\infty)\,, 
\end{equation*}
with $\kappa(d) =0\,.$ These lead to
\begin{align*}
\chi(w) = \int_d^w  \frac{1}{\nu(r)\,  \kappa(r)} \, \dr = \int_d^w \frac{((\kappa(r))^2)'}{\kappa(r)} \, \dr = 2 \int_d^w \kappa'(r) \, \dr =2 \, \kappa(w) \,.  
\end{align*}
Thus, we need to prove that 
\begin{equation}\label{leq4bl}
2 \kappa(r) \int_r^{\infty} \mu(w) \, \kappa(w) \, \dw \leq 4b_{\tu{L}} \quad \forall d<r<\infty\,,
\end{equation}

We denote 
\[\dd h(w):= \int_w^{\infty} \mu(y) \, \dy \quad \forall d<r<\infty \,.\] Since \eqref{BL1mn} means 
\[h(w) \, (\kappa(w))^2 \leq b_{\tu{L}}\,,\] it follows that
\begin{align*}
\int_r^{\infty} \mu(w) \, \kappa(w) \, \dw &= \int_r^{\infty} \sqrt{h(w)} \, \kappa(w) \, \frac{\mu(w)}{\sqrt{h(w)}} \, \dw   \\
&\leq \sqrt{b_{\tu{L}}} \int_r^{\infty} \frac{\mu(w)}{\sqrt{h(w)}} \, \dw = -2 \sqrt{b_{\tu{L}}} \int_r^{\infty} \left(\sqrt{h(w)}\right)' \, \dw =2 \sqrt{b_{\tu{L}}} \sqrt{h(r)}
\end{align*}
Applying this to the left-hand side of \eqref{leq4bl} and together with $\sqrt{h(r)} \, \kappa(r) \leq \sqrt{b_{\tu{L}}}$ \eqref{BL1mn} leads to \eqref{leq4bl} as desired.
\end{proof}

\vspace{20pt}

\noindent \textbf{Acknowledgements.}
The author is deeply grateful to Prof.~Eitan Tadmor for his dedicated time and significant help with this paper (starting from the 2018 High Performance Scientific Computing Conference, VIASM, Hanoi, Vietnam to the QKP2023, IMS, NUS, Singapore) 
with invaluable advices (especially on the conservation laws), and more importantly, inspiring mentoring, that words cannot express.  The author wishes to thank Prof.~Seung-Yeal Ha very much for kindly suggesting this problem (at the Vietnam-Korea Workshop 2017, Duy Tan University, Da Nang, Vietnam).  She is indebted to the sponsoring by Prof.~Yalchin Efendiev and the Institute for Scientific Computation, at Texas A\&M University, College Station, US (TAMU).
Her special gratitude to Prof.~Edriss Titi, who gave her meticulous guides on related aspects of Analysis and PDEs, and enthusiastically brought her attention to the Hardy's inequality, which is then helpfully applied to the paper (at TAMU, 2023).  Also, members of TAMU Math Department and the Department Head Prof.~Peter Howard 
offer a warm hospitality for the author during completing this paper. 
To friends, she genuinely thanks Prof.~Ricardo Alonso (at the 2018 Hyperbolic Problems Conference, Pennsylvania State University, US) for encouraging conversations and thoughtful comments, Prof.~Minh-Binh Tran for motivating meetings and open-minded discussions. 
She really appreciates the seminars held by the analysis groups at Danang University of Education and Duy Tan University, Vietnam.  Her many thanks to the Ki-Net Conferences, US, 2015--2016, and the IMS program Multiscale Analysis and Methods for Quantum and Kinetic Problems (QKP2023), at the National University of Singapore (NUS), with generous support and good colleagues.  

\bibliographystyle{plain} 
\bibliography{ksh}

\begin{thebibliography}{10}

\bibitem{Taylor1}
J.~A. Acebr\'on, L.~L. Bonilla, and R.~Spigler.
\newblock \href{https://doi.org/10.1103/PhysRevE.62.3437}{Synchronization in
  populations of globally coupled oscillators with inertial effects}.
\newblock {\em Phys. Rev. E (3)}, 62(3, part A):3437--3454, 2000.

\bibitem{jose19}
Pedro Aceves-S\'{a}nchez, Mihai Bostan, Jose-Antonio Carrillo, and Pierre
  Degond.
\newblock
  \href{https://www.aimspress.com/article/doi/10.3934/mbe.2019396}{Hydrodynamic
  limits for kinetic flocking models of Cucker--Smale type}.
\newblock {\em Mathematical Biosciences and Engineering}, 16(6):7883--7910,
  2019.

\bibitem{alonsohd}
Ricardo~J. Alonso, Bertrand Lods, and Isabelle Tristani.
\newblock \href{https://arxiv.org/abs/2008.05173}{Fluid dynamic limit of
  Boltzmann equation for granular hard-spheres in a nearly elastic regime},
  2021.
\newblock arXiv:2008.05173.

\bibitem{apostol}
Tom~M. Apostol.
\newblock {\em Mathematical Analysis}.
\newblock Addison Wesley, second edition, 1974.

\bibitem{DL16}
Diogo {Ars{\'e}nio} and Laure {Saint-Raymond}.
\newblock {\em \href{https://ems.press/books/emm/166}{From the
  Vlasov--Maxwell--Boltzmann System to Incompressible Viscous
  Electro-magneto-hydrodynamics}}, volume~1 of {\em EMS Monographs in
  Mathematics}.
\newblock EMS Press, 2019.

\bibitem{BGL}
Claude Bardos, Fran\c{c}ois Golse, and C.~David Levermore.
\newblock \href{https://doi.org/10.1002/cpa.3160460503}{Fluid dynamic limits of
  kinetic equations. \textup{II}. Convergence proofs for the Boltzmann
  equation}.
\newblock {\em Communications on Pure and Applied Mathematics}, 46(5):667--753,
  1993.

\bibitem{e}
L.~L. Bonilla.
\newblock \href{https://doi.org/10.1103/PhysRevE.62.4862}{Chapman--Enskog
  method and synchronization of globally coupled oscillators}.
\newblock {\em Phys. Rev. E}, 62:4862--4868, 2000.

\bibitem{boyer_boch_der}
Franck Boyer and Pierre Fabrie.
\newblock {\em \href{https://doi.org/10.1007/978-1-4614-5975-0}{Mathematical
  Tools for the Study of the Incompressible Navier--Stokes Equations and
  Related Models}}.
\newblock Applied Mathematical Sciences. Springer New York, NY, 2013.

\bibitem{brezis}
Haim Brezis.
\newblock {\em \href{https://doi.org/10.1007/978-0-387-70914-7}{Functional
  Analysis, Sobolev Spaces and Partial Differential Equations}}.
\newblock Springer, 2011.

\bibitem{canizo_muck}
Jos{\'e}~A. Ca{\~n}izo.
\newblock \href{https://canizo.org/page/26}{Muckenhoupt's proof of the Hardy
  inequality in dimension 1}.
\newblock Available online: \url{https://canizo.org/page/26} (accessed in 2023,
  2024).

\bibitem{ff}
Lauren~M. Childs and Steven~H. Strogatz.
\newblock \href{http://dx.doi.org/10.1063/1.3049136}{Stability diagram for the
  forced {K}uramoto model}.
\newblock {\em Chaos}, 18(4):043128, 2008.

\bibitem{2022choihd}
Young-Pil Choi.
\newblock \href{https://doi.org/10.1007/s42985-022-00219-7}{On the rigorous
  derivation of hydrodynamics of the Kuramoto model for synchronization
  phenomena}.
\newblock {\em Partial Differential Equations and Applications}, 4(1):2, 2023.

\bibitem{2018wellposed}
Young-Pil Choi, Seung-Yeal Ha, Qinghua Xiao, and Yinglong Zhang.
\newblock \href{https://doi.org/10.1137/20M1368719}{Asymptotic stability of the
  phase-homogeneous solution to the Kuramoto--Sakaguchi equation with inertia}.
\newblock {\em SIAM Journal on Mathematical Analysis}, 53(3):3188--3235, 2021.

\bibitem{KM1}
Young-Pil Choi, Seung-Yeal Ha, and Seok-Bae Yun.
\newblock \href{http://dx.doi.org/10.1016/j.physd.2010.08.004}{Complete
  synchronization of {K}uramoto oscillators with finite inertia}.
\newblock {\em Phys. D}, 240(1):32--44, 2011.

\bibitem{KM2}
Young-Pil Choi, Seung-Yeal Ha, and Seok-Bae Yun.
\newblock \href{http://dx.doi.org/10.3934/nhm.2013.8.943}{Global existence and
  asymptotic behavior of measure valued solutions to the kinetic
  {K}uramoto--{D}aido model with inertia}.
\newblock {\em Netw. Heterog. Media}, 8(4):943--968, 2013.

\bibitem{cohnm}
Donald~L. Cohn.
\newblock {\em \href{https://doi.org/10.1007/978-1-4614-6956-8}{Measure
  Theory}}.
\newblock Birkh{\"a}user Advanced Texts Basler Lehrb{\"u}cher. Birkh{\"a}user
  New York, NY, second edition, 2013.

\bibitem{PD1}
Pierre Degond, Giacomo Dimarco, and Thi Bich~Ngoc Mac.
\newblock \href{http://dx.doi.org/10.1142/S0218202513400095}{Hydrodynamics of
  the Kuramoto--Vicsek model of rotating self-propelled particles}.
\newblock {\em Math. Models Methods Appl. Sci.}, 24(2):277--325, 2014.

\bibitem{gci1}
Pierre Degond, Amic Frouvelle, Jian-Guo Liu, Sebastien Motsch, and Laurent
  Navoret.
\newblock
  \href{https://proceedings.centre-mersenne.org/articles/10.5802/slsedp.32/}{Macroscopic
  models of collective motion and self-organization}.
\newblock {\em S\'eminaire Laurent Schwartz {\textemdash} EDP et applications},
  pages 1--27, 2012--2013.
\newblock talk:1.

\bibitem{chi1}
Pierre Degond, Jian-Guo Liu, Sebastien Motsch, and Vladislav Panferov.
\newblock \href{http://dx.doi.org/10.4310/MAA.2013.v20.n2.a1}{Hydrodynamic
  models of self-organized dynamics: Derivation and existence theory}.
\newblock {\em Methods Appl. Anal.}, 20(2):89--114, 2013.

\bibitem{hlk2}
Pierre Degond and S\'ebastien Motsch.
\newblock \href{http://dx.doi.org/10.1142/S0218202508003005}{Continuum limit of
  self-driven particles with orientation interaction}.
\newblock {\em Math. Models Methods Appl. Sci.}, 18(supp01):1193--1215, 2008.

\bibitem{12circlepath}
Florian D\"{o}rfler and Francesco Bullo.
\newblock \href{https://doi.org/10.1137/10081530X}{On the critical coupling for
  Kuramoto oscillators}.
\newblock {\em SIAM Journal on Applied Dynamical Systems}, 10(3):1070--1099,
  2011.

\bibitem{12circle}
Florian D\"{o}rfler and Francesco Bullo.
\newblock \href{https://doi.org/10.1137/110851584}{Synchronization and
  transient stability in power networks and nonuniform Kuramoto oscillators}.
\newblock {\em SIAM Journal on Control and Optimization}, 50(3):1616--1642,
  2012.

\bibitem{evans}
Lawrence~C. Evans.
\newblock {\em \href{https://bookstore.ams.org/gsm-19-r}{Partial Differential
  Equations: Second Edition}}, volume~19 of {\em Graduate Studies in
  Mathematics}.
\newblock American Mathematical Society, 2010.

\bibitem{evans15}
L.C. Evans and R.F. Gariepy.
\newblock {\em \href{https://doi.org/10.1201/b18333}{Measure Theory and Fine
  Properties of Functions}}.
\newblock Chapman and Hall/CRC, first edition, 2015.

\bibitem{folland}
Gerald~B. Folland.
\newblock {\em
  \href{https://www.wiley.com/en-us/Real+Analysis:+Modern+Techniques+and+Their+Applications,+2nd+Edition-p-9780471317166}{Real
  Analysis: Modern Techniques and Their Applications}}.
\newblock John Wiley \& Sons, Inc., New York, second edition, 1999.

\bibitem{frank22}
R.~L. Frank, A.~Laptev, and T.~Weidl.
\newblock \href{https://doi.org/10.1007/s10958-022-06199-8}{An improved
  one-dimensional Hardy inequality}.
\newblock {\em Journal of Mathematical Sciences}, 268(3):323--342, 2022.

\bibitem{expand1}
Amic Frouvelle.
\newblock \href{http://dx.doi.org/10.1142/S021820251250011X}{A continuum model
  for alignment of self-propelled particles with anisotropy and
  density-dependent parameters}.
\newblock {\em Math. Models Methods Appl. Sci.}, 22(7):1250011, 40, 2012.

\bibitem{papa_bochner}
Leszek Gasinski and Nikolaos~S. Papageorgiou.
\newblock {\em \href{https://doi.org/10.1201/9781420035049}{Nonlinear
  Analysis}}.
\newblock Chapman and Hall/CRC, first edition, 2005.

\bibitem{golsehd}
Fran{\c{c}}ois Golse.
\newblock \href{https://doi.org/10.1016/S1874-5717(06)80006-X}{Chapter 3 -- The
  Boltzmann Equation and Its Hydrodynamic Limits}.
\newblock In C.M. Dafermos and E.~Feireisl, editors, {\em Handbook of
  Differential Equations Evolutionary Equations}, volume~2 of {\em Handbook of
  Differential Equations: Evolutionary Equations}, pages 159--301.
  North-Holland, 2005.

\bibitem{GJV}
Thierry Goudon, Pierre-Emmanuel Jabin, and Alexis Vasseur.
\newblock \href{http://dx.doi.org/10.1512/iumj.2004.53.2509}{Hydrodynamic limit
  for the {V}lasov--{N}avier--{S}tokes equations. {II}. {F}ine particles
  regime}.
\newblock {\em Indiana Univ. Math. J.}, 53(6):1517--1536, 2004.

\bibitem{Kramers1}
Shamik Gupta, Alessandro Campa, and Stefano Ruffo.
\newblock
  \href{https://link.aps.org/doi/10.1103/PhysRevE.89.022123}{Nonequilibrium
  first-order phase transition in coupled oscillator systems with inertia and
  noise}.
\newblock {\em Phys. Rev. E}, 89:022123, 2014.

\bibitem{KM3}
Seung-Yeal Ha, Dongnam Ko, Jinyeong Park, and Xiongtao Zhang.
\newblock \href{http://dx.doi.org/10.4171/EMSS/17}{Collective synchronization
  of classical and quantum oscillators}.
\newblock {\em EMS Surv. Math. Sci.}, 3(2):209--267, 2016.

\bibitem{2020difflim}
Seung-Yeal Ha, Woojoo Shim, and Yinglong Zhang.
\newblock \href{https://doi.org/10.1137/19M1237454}{A diffusion limit for the
  parabolic Kuramoto--Sakaguchi equation with inertia}.
\newblock {\em SIAM Journal on Mathematical Analysis}, 52(2):1591--1638, 2020.

\bibitem{2019robust}
Seung‐Yeal Ha, Jaeseung Lee, and Yinglong~J. Zhang.
\newblock \href{https://doi.org/10.1090/qam/1533}{Robustness in the instability
  of the incoherent state for the Kuramoto--Sakaguchi--Fokker--Planck equation
  with frustration}.
\newblock {\em Quarterly of Applied Mathematics}, 77(3):631--654, 2019.

\bibitem{Hardy1920}
G.~H. Hardy.
\newblock \href{https://doi.org/10.1007/BF01199965}{Note on a theorem of
  Hilbert}.
\newblock {\em Mathematische Zeitschrift}, 6(3):314--317, 1920.

\bibitem{trivisa14}
Trygve~K. Karper, Antoine Mellet, and Konstantina Trivisa.
\newblock \href{https://doi.org/10.1142/S0218202515500050}{Hydrodynamic limit
  of the kinetic Cucker--Smale flocking model}.
\newblock {\em Mathematical Models and Methods in Applied Sciences},
  25(01):131--163, 2015.

\bibitem{lancellotti}
Carlo Lancellotti.
\newblock \href{https://doi.org/10.1080/00411450508951152}{On the Vlasov limit
  for systems of nonlinearly coupled oscillators without noise}.
\newblock {\em Transport Theory and Statistical Physics}, 34(7):523--535, 2005.

\bibitem{leonipde}
Giovanni Leoni.
\newblock {\em \href{https://bookstore.ams.org/view?ProductCode=GSM/181}{A
  First Course in Sobolev Spaces}}, volume 181 of {\em Graduate Studies in
  Mathematics}.
\newblock American Mathematical Society, second edition, 2017.

\bibitem{leonifrac}
Giovanni Leoni.
\newblock {\em
  \href{https://www.ams.org/publications/authors/books/postpub/gsm-229}{A First
  Course in Fractional Sobolev Spaces}}, volume 229 of {\em Graduate Studies in
  Mathematics}.
\newblock American Mathematical Society, 2023.

\bibitem{Levermore}
C.~David Levermore.
\newblock \href{http://dx.doi.org/10.1080/03605309308820972}{Entropic
  convergence and the linearized limit for the Boltzmann equation}.
\newblock {\em Communications in Partial Differential Equations},
  18(7-8):1231--1248, 1993.

\bibitem{tm_entropic}
Tina Mai.
\newblock \href{https://arxiv.org/abs/1612.05096}{Entropic convergence and the
  linearized limit for the Boltzmann equation with external force}, 2016.
\newblock arXiv:1612.05096.

\bibitem{Muck1972}
Benjamin Muckenhoupt.
\newblock \href{http://eudml.org/doc/217718}{Hardy's inequality with weights}.
\newblock {\em Studia Mathematica}, 44(1):31--38, 1972.

\bibitem{hardyopic}
Bohum\'{i}r Opic and Alois Kufner.
\newblock {\em Hardy-type Inequalities}.
\newblock John Wiley \& Sons, New York, 1990.

\bibitem{2010benoit-syn}
Khashayar Pakdaman, Beno{\^i}t Perthame, and Delphine Salort.
\newblock \href{https://dx.doi.org/10.1088/0951-7715/23/1/003}{Dynamics of a
  structured neuron population}.
\newblock {\em Nonlinearity}, 23(1):55, 2009.

\bibitem{royden88}
Halsey Royden.
\newblock {\em {Real analysis}}.
\newblock Macmillan, New York, third edition, 1988.

\bibitem{L}
Laure Saint-Raymond.
\newblock {\em \href{https://doi.org/10.1007/978-3-540-92847-8}{Hydrodynamic
  limits of the Boltzmann equation}}.
\newblock Lecture Notes in Mathematics. Springer Berlin, Heidelberg, 1 edition,
  2009.

\bibitem{12plane}
Y.~{Wang} and F.~J. {Doyle}.
\newblock \href{https://ieeexplore.ieee.org/document/6287554}{Exponential
  synchronization rate of Kuramoto oscillators in the presence of a pacemaker}.
\newblock {\em IEEE Transactions on Automatic Control}, 58(4):989--994, 2013.

\end{thebibliography}
\end{document}